\documentclass{article}

\usepackage[utf8]{inputenc}
\usepackage{color}

\usepackage{amsthm,amsmath,dsfont,enumerate}

\usepackage[T1]{fontenc}
\usepackage[sc]{mathpazo}

\usepackage{amsfonts,amssymb,graphics,bm,latexsym,amsbsy,rotating,tabularx,url,hyperref,bbm,comment}

\usepackage[left=2.5cm,right=2.5cm,top=2cm,bottom=1cm,includeheadfoot,asymmetric]{geometry} 

\usepackage{comment}

\usepackage{natbib}

\newcommand{\R}{\mathds{R}}
\newcommand{\N}{\mathds{N}}

\newcommand{\E}{{\rm {\mathbb E}}}
\newcommand{\var}{{\rm {\mathbb V}ar}}
\newcommand{\cov}{{\rm {\mathbb C}ov}}

\newcommand{\mvn}{f^{MvN}_d}
\newcommand{\mg}{f^{MG}_d}
\newcommand{\eqd}{\stackrel{d}{=}}

\def\bone{{\mathds 1}}

\newcommand{\si}{{\sigma}}

\newcommand{\tto}{t\to\infty}

\newtheorem{theorem}{Theorem}[section]

\newtheorem{proposition}[theorem]{Proposition}
\newtheorem{lemma}[theorem]{Lemma}
\newtheorem{definition}[theorem]{Definition}

\newtheorem{remark}[theorem]{Remark}

\newcommand{\beao}{\begin{eqnarray*}}
\newcommand{\eeao}{\end{eqnarray*}\noindent}

\newcommand{\beam}{\begin{eqnarray}}
\newcommand{\eeam}{\end{eqnarray}\noindent}

\newcommand{\barr}{\begin{array}}
\newcommand{\earr}{\end{array}}

\allowdisplaybreaks[4]

\numberwithin{equation}{section} 

\begin{document}

\thispagestyle{empty}

\fontsize{12}{14pt plus.8pt minus .6pt}\selectfont \vspace{0.8pc}

\begin{center}
 {\large \textbf{FRACTIONALLY INTEGRATED COGARCH PROCESSES} 	}
 \end{center}

\fontsize{11}{12pt plus.8pt minus .6pt}\selectfont \vspace{.4cm}

 \centerline{STEPHAN HAUG, CLAUDIA KL\"UPPELBERG AND GERMAN STRAUB} 
\vspace{.4cm} \centerline{\it Technical University of Munich} \vspace{.55cm} \fontsize{10}{11.5pt plus.8pt minus
.6pt}\selectfont

\begin{quotation}
\noindent \textit{ Abstract:}
{We construct fractionally integrated continuous-time GARCH models, which capture the observed long range dependence of squared volatility in high-frequency data. 
Since the usual Molchan-Golosov and Mandelbrot-van-Ness fractional kernels lead to problems in the definition of the model,
we resort to moderately long memory processes by choosing a fractional parameter $d\in(-0.5,0)$ and remove the singularities of the kernel to obtain non-pathological sample paths. 
The volatility of the new fractional COGARCH process has positive features like stationarity, and its covariance function shows an algebraic decay, which make it applicable to econometric high-frequency data. In an empirical application the model is fitted to exchange rate data using a simulation-based version of the generalised method of moments.}

\vspace{9pt}
\end{quotation}



\section{Introduction}

Long range dependence in time series or their latent volatilities has been observed in many application areas and is often modelled by fractional processes.
From a statistical point of view, using long range dependence models for data or their latent structure, the question should be asked, where this effect originates from. 
Obviously, a small deterministic trend or seasonality introduces long range dependence into the data. But also, as is well-known (cf. \cite{mikoschstarica}, \cite{laixing} and \cite{DavisYau}), change points in the model or a change of parameters in time can be responsible for long range dependence effects. 
On the other hand, when these effects are small, and trends or parameters change very slowly in time, they are statistically often not detectable; cf. \citet{Dette} for test procedures and a literature review on this topic.
This is in particular true, when such small effects occur in the volatility process, which is latent. 
Consequently, a long range dependent model may well be the best choice as it at least captures the observed behaviour.

Prominent discrete-time models
in finance have been early on discussed for instance in \citet*{Baillie199602}, \citet*{Baillie1996}, \cite{Bollerslev1996}, \citet*{Ding1993} and \cite{GP}.

When it comes to continuous-time stochastic volatility processes long memory Gaussian models for the log-volatility have been suggested and analysed in  
\cite{Comte1996}, \cite{ComteRenault1996}, and \cite{ComteRenault1998}. 
As Gaussian models are distributionally restricted, such concepts have been extended to more general models in two ways. 
\citet*{Comte2012} introduce long memory in the Heston model by considering an affine volatility model, where they fractionally integrate the square-root volatility process. 
As they emphasize, there is nowadays a quite general agreement on the idea that jump components in
the return process (and possibly in the volatility process itself) are needed to explain
very short term option prices, but long memory processes can address option
pricing puzzles as steep volatility smiles in long term options and co-movements
between implied and realized volatility for longer maturities without introducing unrealistic volatility behavior
in both short and long term returns (cf. \cite{Comte2012} for more extensive discussions).

Another line of research involves general L\'evy processes instead of Brownian motion. 
For instance, \citet*{anh2002} proposed a L\'evy driven stochastic volatility model, where the volatility process is of moving average type and allows for long memory by the choice of the moving average kernel.  
Another related approach replaces the fractional Brownian motion driving process by a fractional L\'evy process, which have been introduced as analogues to fractional Brownian motion in \citet*{Marquardt2006}.
This leads to fractional Ornstein-Uhlenbeck processes and other continuous-time long memory models as in \citet{FK}.


Introducing long range dependence in non-linear models like ARCH and GARCH and their continuous-time counterparts resulting in stationary models is much more difficult, and we discuss this first for discrete-time models.
Different approaches have been proposed in the literature for defining long range dependent ARCH and GARCH models.
 \cite{Baillie1996} and \cite{Ding1996} were among the first to have discussed ways to incorporate long memory in a stationary ARCH model. 
But some of these approaches have certain drawbacks as emphasised in \cite{Mikosch2003}.
Indeed,  \citet*{Douc2008} were the first to specify conditions for the existence of a strictly stationary solution to the fractionally integrated GARCH$(p,d,q)$ equation.
But they give a proof only for the case $p=q=0$ and certain values of $d\in (d^\ast,1)$ for $d^\ast >0$, where  
they use the fact that fractionally integrated GARCH$(0,d,0)$ models are a subclass of ARCH($\infty$) models. 
The conditions derived in \citet*{Kaza2003} for the existence of a strictly stationary solution of an ARCH($\infty$) model rule out the fractionally integrated GARCH model.
The model was extended by \cite{Douc2008}, however, the second moment of the resulting stationary solution is not finite.
\citet*{Giraitisetal2015} showed the existence of a stationary FIGARCH process with finite variance by considering it as solution of an ARCH($\infty$) representation without constant term in the defining equation.  
They also proved that the volatility process in this model has a covariance function which is non-summable and, hence, possesses a long memory property. 


 The goal of this paper is to construct a fractionally integrated continuous-time GARCH model in order to capture long range dependence of squared volatility as observed in high-frequency data. 
Our approach to obtain a new continuous-time non-linear stochastic volatility model 
is based on the continuous-time GARCH (COGARCH) process. 
The definition of the COGARCH process has been based on replacing the driving noise of a GARCH model by the jumps of a L\'evy process.
In order to define a fractionally integrated COGARCH process, we cannot use the fractionally integrated GARCH model as a sole guidance because of the problems discussed above. However, it will be helpful to review the main ideas of the construction of the COGARCH process before defining a new fractionally integrated version.

\subsection{Review of continuous-time GARCH models}

The GARCH(1,1) model is a discrete-time process with three parameters, $a_0>0$,  $a_1\ge 0$, $b_1\ge 0$, specifying the variance as a  discrete-time stochastic recursion, or difference equation.
We write it using two equations, one specifying the mean level process (the observed data, perhaps after removal of trend or other
deterministic feature, to approximate stationarity) and the other specifying the variance process, which is time dependent and randomly fluctuating. Thus, 
\begin{equation} \label{1a} 
Y_i=\varepsilon_i \sigma_i\,,\quad i=1,2,\ldots\,,
\end{equation}
with squared volatility
\begin{equation} \label{1b}
\sigma^2_i=a_0 + a_1 Y_{i-1}^2 +b_1 \sigma^2_{i-1}
=a_0 + (a_1 \varepsilon_{i-1}^2 +b_1 )\sigma^2_{i-1}\,,\quad i=1,2,\ldots\,.
\end{equation}
Here the starting values $\varepsilon_0$ and $\sigma_0$ are given quantities, possibly random, and usually assumed independent of  $(\varepsilon_i)_{i=1,2,\ldots}$, which are the sole source of randomness in the model.
The $\varepsilon_i$, $i=1,2,\ldots$, are assumed to be independent identically distributed (i.i.d.) random variables with mean 0.
Serial dependence between the $Y_i$ is introduced via the dependence of the $\sigma^2_i$ on their past values.
Conditional on $\sigma_i$,  $Y_i$ simply has the distribution of $\varepsilon_i$, scaled by  $\sigma_i$, which in general (as long as $a_1,b_1>0$) is time dependent, hence the conditional heteroscedasticity part of the terminology.
The ``autoregressive" aspect refers to the form of the dependence of  $\sigma^2_i$ on  $\sigma^2_{i-1}$.

\subsubsection{Continuous-time limits of  GARCH models}

Motivated by the availability of high-frequency data and by a need for option pricing technologies for realistic models,
classical diffusion limits have been used in a natural way to obtain continuous-time limits of discrete-time processes.
For the GARCH(1,1) model, the best-known limit model is the volatility-modulated model due to \cite{Nelson1990b} given by 
\begin{equation*}
    \mathrm{d} G_t = \sigma_t {\rm d}B_t^{(1)}, \qquad t \ge 0,
\end{equation*}
where the squared volatility process $(\sigma^2_t)_{t\ge0}$ is the solution of the stochastic differential equation (SDE)
\begin{equation*}
\mathrm{d} \sigma^2_t = \left(\alpha_0 -\beta_1\sigma^2_t\right) \mathrm{d} t +\alpha_1
\sigma^2_t\mathrm{d} B_t^{(2)}, \qquad t > 0,
\end{equation*}
with initial values $G_0$ and $\si_0^2$, $B^{(1)}$ and $B^{(2)}$ are independent Brownian motions, and
parameters $\alpha_0 > 0$, $\alpha_1\ge 0,\beta_1\ge 0$.

Unfortunately, diffusion limits can lose certain essential properties of the discrete-time models.
It is surprising and counter-intuitive, for example,
that Nelson's diffusion limit of the GARCH process is driven by two
independent Brownian motions, i.e., has two independent sources of
randomness, whereas the discrete-time GARCH process is driven only by a
single i.i.d. noise sequence. One of the features of the GARCH process
is the idea that large innovations in  the price process are almost
immediately manifested as innovations in the volatility process, but this feedback mechanism is lost in models such as Nelson's continuous-time version.
Further, the appearance of an extra source of variation can have
implications for completeness of the model when used for option pricing.

The phenomenon that a diffusion limit may be driven by two
independent Brownian motions, while the discrete-time model is given in
terms of a single i.i.d. sequence, is not restricted to the
classical GARCH process. \cite{Duan1997} has shown that
this occurs for many GARCH like processes.

Moreover, such continuous-time limits can have distinctly different
statistical properties compared to the original discrete-time processes.
As was shown in \cite{Wang2002}, parameter estimation in the discrete-time
GARCH and the corresponding
continuous-time limit stochastic volatility model may
yield different estimates.
Thus these  kinds of continuous-time
models are probabilistically and statistically different
from their discrete-time progenitors.
See \cite{Lindner2009} for an overview of  continuous-time
approximations to GARCH processes.

In \citet*{Klueppelberg2004}, the authors proposed a radically
different approach to obtain a continuous-time model. 
The COGARCH(1,1) model is a direct analogue of the
discrete-time GARCH(1,1), based on a single
background driving L\'evy process,
and generalises the essential features
of the discrete-time GARCH(1,1) process in a natural way.
In the next subsection we review this model.

\subsubsection{The COGARCH(1,1) model}

The COGARCH(1,1) model is specified by two equations, the mean and variance equations, analogous to \eqref{1a} and \eqref{1b}.
The single source of randomness is a so-called {\em background driving L\'evy process}
$L:=(L_t)_{t\ge 0}$ with characteristic triple $(\gamma_L,\tau^2_L,\nu_L)$, where $\gamma_L$ is a real-valued parameter, $\tau^2_L\geq 0$ is the variance of the Brownian component of $L$ and $\nu_L$ is the L\'evy measure, which
dictates how the jumps occur. 
We refer to \cite{ContTankov2003} for background on L\'evy processes in the context of financial modeling.

The driving noise of the GARCH(1,1) model are the i.i.d. innovations $\varepsilon_i$. In the
COGARCH model these innovations are replaced by  the jumps  $\Delta L_s = L_s-L_{s-}$ of $L$, where $L_{s-}$ is the left limit of the sample path of $L$ at $s>0$. 
Thus the  COGARCH(1,1) process  $(G_t)_{t\ge 0}$ is given by
\begin{equation*}
    \mathrm{d} G_t = \sigma_{t-} {\rm d}L_t, \qquad t > 0,
\end{equation*}
where the squared volatility process $(\sigma^2_t)_{t\ge0}$ is  the solution of the SDE
\begin{equation}\label{2b}
\mathrm{d} \sigma^2_t = \left(\alpha_0\beta_1 -\beta_1\sigma^2_{t-}\right) \mathrm{d} t +\alpha_1
\sigma^2_{t-}\mathrm{d} [L,L]^{(d)}_t, \qquad t > 0,
\end{equation}
for parameters $\alpha_0 > 0$, $\alpha_1\ge 0$  and $\beta_1\ge 0$ and initial values $G_0$ and $\si^2_0$.
Here $([L,L]^{(d)}_t)_{t\geq 0}$ denotes the discrete part of the quadratic variation process of $L$, which is defined as
 \[
 [L,L]_t = \sigma^2 t+ \sum_{0<s\leq t} (\Delta L_s)^2 = \sigma^2 t + [L,L]^{(d)}_t\,,\qquad t\ge0.
 \]
$([L,L]^{(d)}_t)_{t\ge 0}$ is an example for a subordinator, i.e., a process with non-negative, independent and stationary increments, in particular, it has increasing sample paths. A formal definition can be found at the beginning of Section~2.1.

To see the analogy with  \eqref{1a} and \eqref{1b},
note from  \eqref{1b} that
 \[
\sigma^2_i - \sigma^2_{i-1}
=a_0 - (1-b_1)  \sigma^2_{i-1} +
a_1 \sigma^2_{i-1}\varepsilon^2_{i-1},
 \]
which corresponds to \eqref{2b}  when the time increment $\mathrm{d} t$ is taken as a unit time interval.


The solution of the SDE \eqref{2b} can be written with the help of an auxiliary L\'evy process $(X_t)_{t\ge 0}$ defined by
 \begin{equation*}
 X_t = \beta_1 t-\sum_{0<s\le t}\log(1+\alpha_1(\Delta L_s)^2), \qquad t \ge 0.
 \end{equation*}
The process $(X_t)_{t\ge 0}$  is a L\'evy process with positive drift and negative jumps; i.e., its L\'evy measure has support $(-\infty,0)$. 
Further it is a process of finite variation, arising in a natural way in  \cite{Klueppelberg2004},
where the COGARCH(1,1) is motivated directly as an analogue
to the discrete-time GARCH(1,1) process.

Using Ito's formula (see e.g. \citet[Section 8.3]{ContTankov2003}), it can be verified that the solution of  \eqref{2b} can be written in terms of $(X_t)_{t\ge0}$ as
 \begin{equation*}
\sigma^2_t=
e^{-X_t}\left(\alpha_0\beta_1\int_0^t e^{X_s}{\rm d}s+\sigma^2_0\right),\qquad t\ge 0,
\end{equation*}
which  reveals $(\sigma^2_t)_{t\ge0}$ as a generalised Ornstein-Uhlenbeck  process,
parameterised by $(\alpha_0, \alpha_1, \beta_1)$,
and driven by the L\'evy process $L$, see \cite{LindnerMaller2005} for details on
generalised Ornstein-Uhlenbeck processes.

\citet[Theorem 3.2]{Klueppelberg2004} show that the variance process
$(\sigma^2_t)_{t\ge 0}$ for the COGARCH(1,1) is a time homogeneous
Markov process, and, further, that the bivariate process
$(G_t,\sigma^2_t)_{t\ge 0}$ is  Markovian.
A finite random variable $\sigma^2_\infty$ exists as limit
in distribution of  $\sigma^2_t$ as $t\to\infty$, if the cumulant generating function 
\[
    \Psi(s) =\ln(\E(e^{-sX_1})) = -\beta_1 s +\int_\R\big((1+\alpha_1 x^2)^s-1)\nu_L(dx)\,,
\]
of the auxiliary process $(X_t)_{t\ge 0}$ satisfies $\Psi(1)<0$.
When $\sigma^2_\infty$ exists, it has the same distribution as $\alpha_0\beta_1 \int_0^\infty e^{-X_t}\mathrm{d} t$.
If this is the case and  $(\sigma^2_t)_{t\ge 0}$ starts with $\sigma^2_0$ having the distribution of
$\sigma^2_\infty$, independent of $L$, then $(\sigma^2_t)_{t\ge 0}$ is
strictly stationary and  $(G_t)_{t\ge 0}$ is a process
with  stationary increments, see \citet[Theorem 3.2 and Corollary 3.1]{Klueppelberg2004}.

Returns over time intervals of fixed length $\Delta>0$ are denoted by
\[
    \Delta G_t = G_t-G_{t-\Delta}=\int_{(t-\Delta,t]}\si_{s-}\,\mathrm{d}L_s, \quad t \geq \Delta\,,
\]
so that $(\Delta G_{\Delta i})_{i \in \mathbb{N}}$ describes a sequence of returns over equidistant and non-overlapping time intervals.
Calculating the corresponding quantities for the volatility yields
\begin{align}\label{volact}
\Delta\si^{2}_{\Delta i} 
    & =  \si^{2}_{\Delta i}-\si^{2}_{\Delta(i-1)} 
      = \int_{(\Delta(i-1),\Delta i]}\left((\alpha_0\beta_1 -\beta_1\si^2_s)\,\mathrm{d}s+\alpha_1\,
\si_{s-}^2\, d[L,L]_s^{(d)}\right)\nonumber\\
    & =   \alpha_0\beta_1 \Delta -\beta_1 \int_{(\Delta(i-1),\Delta i]}\si^2_s\,\mathrm{d}s + \alpha_1 \int_{(\Delta(i-1),\Delta i]}\si_{s-}^2\, d[L,L]^{(d)}_s\,,\quad i=0,1,2,\dots\,.
\end{align}
It is also worth noting that the stochastic process $([G,G]_t)_{t\ge0}$ defined as
 \[
[G,G]_t   :=  \int_{(0,t]} \sigma^2_{s-}\,d[L,L]_s 
= \sigma^2 \int_0^t \sigma_{s-}\mathrm{d}s + \sum_{0<s\le t}\sigma^2_{s-}(\Delta L_s)^2, \quad t\ge 0\,,
 \]
is the quadratic variation of $G$, so that the second integral
in \eqref{volact} corresponds to the
jumps of the quadratic variation of $G$ accumulated during $(\Delta(i-1),\Delta i]$.

The following result (see \citet*[Proposition 1]{Haugetal2007}) shows that the COGARCH(1,1) has a similar moment structure as
the GARCH(1,1) model. In particular, there is no correlation between increments, but between the squared increments.

Suppose that $L$ has finite variance and zero mean, and that
$\Psi(1)<0$. Let $(\sigma_t^2)_{t \geq 0}$ be the stationary volatility
process, so that $(G_t)_{t \geq 0}$ has stationary increments. Then
$\E (G_t^2) < \infty$ for all $t \geq 0$, and
for every $t, h\geq \Delta>0$ it holds that
\[
\E(\Delta G_t) \, = \, 0\,,\quad
\E(\Delta G_t)^2 \, = \, \frac{\alpha_0\beta_1 \Delta}{|\Psi(1)|}\E(L_1^2)\,,\quad
\cov(\Delta G_t,\Delta G_{t+h}) \, = \,  0.
\]
If further $\E(L_1^4)<\infty$ and $\Psi(2)<0$, then
$\E (G_t^4) < \infty$ for all $t \geq 0$ and, if additionally
the L\'evy measure $\nu_L$ of $L$ is such that
$\int_{\R}x^3\nu_L(dx)=0$, then for every $t, h \geq \Delta > 0$ it holds that
\begin{align*}
\E(\Delta G_t)^4  
    & = 6 \E (L_1^2) \frac{ (\alpha_0\beta_1)^2}{\Psi(1)^2}
\left(\frac{2\beta_1}{\varphi} + 2 \tau_L^2 - \E (L_1^2)\right)
\left( \frac{ 2}{|\Psi(2)|}- \frac{1}{|\Psi(1)|} \right) 
\left( \Delta -\frac{1 - e^{-\Delta |\Psi(1)|}}{|\Psi(1)|} \right)\\
& \quad 
+ \frac{2(\alpha_0\beta_1)^2}{\varphi^2}
\left( \frac{2}{|\Psi(2)|}
- \frac{1}{|\Psi(1)|} \right) \Delta
  + 3 \frac{(\alpha_0\beta_1)^2}{\Psi(1)^2} (\E (L_1^2))^2 \Delta^2
\end{align*}
and
\begin{align*}
\cov((\Delta G_t)^2,(\Delta G_{t+h})^2)
    & =   \E (L_1^2)\frac{(\alpha_0\beta_1)^2}{|\Psi(1)|^3}
\left(\frac{2 \beta_1}{\varphi}+ 2\tau_L^2 - \E (L_1^2)\right)
\left( \frac{2}{|\Psi(2)|} - \frac{1}{|\Psi(1)|}
\right)\\ 
    &\quad \times \left(1-e^{-\Delta |\Psi(1)|}\right) \left( e^{\Delta |\Psi(1)|} -1 \right)
  \;  e^{-h |\Psi(1)|} \, > \, 0\,.
\end{align*}
Due to its exponential decay, the autocovariance function of the squared returns shows that the COGARCH(1,1) is a short-memory process. This fact is the main motivation for our paper. 
The goal is to define a continuous-time GARCH process with slowly decaying autocovariance function of the squared returns.
At the heart of the definition of the new fractionally integrated COGARCH model are so-called fractional L\'evy processes, which are introduced in the next section.

\subsection{Fractional L\'evy processes}

Fractional L\'evy processes (fLp) generalise fractional Brownian motion (fBm) in a natural way.
The fBm has a long history, for instance,
it is well-known that fBm can be defined as a stochastic integral of a Volterra-type kernel with respect to Brownian
motion.
Two such kernels with fractional parameter $d\in (-0.5,0.5)$ are the Mandelbrot-van-Ness kernel, 
which leads to fBm on $\R$, and the Molchan-Golosov kernel, which results in fBm on $\R_+=[0,\infty)$, 
see e.g. Chapter~3 in \citet{mishura2018}.
Such Gaussian processes have continuous sample paths, stationary increments, and they are self-similar. 
Moreover, they can model long range dependence for $d>0$.

For fractional parameter $d\in (-0.5,0.5)$  the Molchan-Golosov (MG) kernel is defined for $t\ge0$ as
\beam\label{mgkernel}
\mg(t,s) = \left(\frac{(2d+1)\Gamma(1-d)}{\Gamma(1+d)\Gamma(1-2d)}\right)^{\frac{1}{2}} (t-s)^{d} {}_2F_1\left(-d, d, d+1, \frac{s-t}{s}\right) \,,\quad s\in[0,t],
\eeam
where ${}_2F_1$ is Gauss' hypergeometric function (see e.g. \citet{Olde2010}), and the Mandelbrot-van-Ness (MvN) kernel is defined for $t\in\R$ as
\beam\label{mvnkernel}
 \mvn(t,s) = \frac{1}{\Gamma(d+1)} \left((t-s)^d_+ - (-s)^d_+\right), \qquad s\in\R,
\eeam
where $x_+:=\mathrm{max}\{x,0\}$ for $x\in\R$. 
The latter kernel allows for the following representation of fBm
 \beam\label{fBm}
 B^d_t =  \int_\R \mvn(t,s) dW_s\,,\qquad t\in \R\,,
 \eeam
where $(W_t)_{t\in\R}$ is a two-sided Brownian motion. 
Replacing the driving Brownian motion by a non-Gaussian L\'evy process leads to fractional models with a wealth of finite-dimensional distributions. 

Both models 
\beam\label{fracsub}
\tilde L_t^d = \int_0^t  \mg(t,s) dL_s,\,\, t\ge 0,\quad\mbox{and}\quad L_t^d =\int_{-\infty}^\infty \mvn(t,s) dL_s,\,\, t\in\R,
\eeam
are useful depending on the envisaged application.
For instance, the MG-fLp on the left-hand side has no infinite history, whereas the MvN-fLp on the right-hand side has.
More properties have been shown e.g. in  \citet*{Bender2012}, \citet*{Engelke2013}, \citet*{Marquardt2006}, and \citet*{Tikan2011}. 

In the COGARCH(1,1) model the driving L\'evy process of the squared volatility process \eqref{2b} is the subordinator $([L,L]^{(d)}_t)_{t\geq 0}$, which is the discrete part of the quadratic variation of the background driving L\'evy process $L$. 
As a consequence of \citet*[Theorem~2.7]{Rajput1989},  
the existence of a stochastic integral of a kernel $f$ with respect to an arbitrary subordinator $(S_t)_{t\in\R}$ requires that $(S_t)_{t\in\R}$ has finite second moment and $f\in L^1\cap L^2$. 
Unfortunately, the MG kernel, which belongs to $L^1\cap L^2$, leads to a fractional subordinator which has not necessarily stationary increments; cf. \citet[Proposition~3.11]{Tikan2011}. 
On the other hand, by \citet[Proposition~2]{Engelke2013}, the MvN kernel belongs to $L^1\cap L^2$ only for negative fractional parameter, whereas a long memory property is obtained only for $d\in (0,0.5)$. 
In addition, for negative $d$ the kernel has singularities, which leads with positive probability to discontinuous and unbounded sample paths of the fLp; cf. \citet[Theorem 4]{Rosinski1989}.
With the goal to obtain non-pathological sample paths we define a modified MvN kernel and obtain a fractional subordinator, which has a continuous modification and stationary increments. 
The autocovariance function of these increments decreases with an algebraic rate. 
As it is integrable, the increments do not show long memory behaviour in the classical sense, but for $d$ close to zero, their autocovariance function decreases at a very slow rate.
This allows us to define a fractionally integrated COGARCH(1,1) volatility process driven by the subordinator $([L,L]^{(d)}_t)_{t\ge0}$. 
We also find a stationary version of this new volatility model, which results in a  volatility-modulated process with stationary increments.

\bigskip

Our paper is organised as follows. 
In Section~\ref{s2} we define a fractionally integrated COGARCH(1,1) process. 
For its definition we introduce a new fractional subordinator based on a modified MvN kernel,
which drives the volatility process. An extension to higher order fractionally integrated COGARCH($p,q$) models is straightforward.
Our main result in Section~\ref{s2} is the existence of a stationary version of the variance process, which implies that the volatility-modulated process has stationary increments.  
For a statistical application, we introduce  in Section \ref{s3}  a simulation based generalised method of moment estimator. 
The finite sample behaviour of the proposed estimator is analysed in a small simulation study. 
Afterwards we fit the model to log-returns of two different exchange rate data. 
In Section~\ref{s4} we present properties of the modified MvN kernel and those fractional subordinators, which are defined as an integral transform of a subordinator with the new modified MvN kernel. 
In particular, we show that the covariances of non-overlapping increments decrease algebraically.
All proofs are summarised in an Appendix.

\section{The fractionally integrated COGARCH process}\label{s2}

The driving L\'evy process of the volatility process in the COGARCH(1,1) model in \eqref{2b} is the discrete part of the quadratic variation process $([L,L]_t^{(d)})_{t\geq 0}$ of $L$. 
To introduce a long memory property, we will use the fractional subordinator,
\begin{align}\label{eq:fracSvola}
	S^{a,d}_t = \int_\R f_{a,d}(t,u)\; \mathrm{d}[L,L]^{(d)}_u, \qquad t\geq 0,
\end{align}
as driving process of the quadratic volatility, where $f_{a,d}$ is the modified MvN kernel, given for  $t\ge 0$, $d\in(-0.5,0)$ and $a>0$ as
 \beao
 f_{a,d}(t,s) 
    = \left(a+(-s)_+\right)^d - \left(a+(t-s)_+\right)^d, \qquad s\in\R.
 \eeao
The choice of $d$ as well as the kernel  $f_{a,d}$ and 
the fractional subordinator, defined as an integral transform of an arbitrary subordinator with $f_{a,d}$, are discussed in more detail in Section~\ref{s4}. 

To define the fractional subordinator \eqref{eq:fracSvola} as integral over the whole of $\R$, we need $[L,L]^{(d)}_t$ to be defined for all  $t\in\R$. 
We follow the usual procedure and define for any L\'evy process $L$ a two-sided version as
\beam\label{2L}
L_t=-L^{(1)}_{-t-}\bone_{\{t<0\}}+L^{(2)}_t\bone_{\{t\ge 0\}}\,,\quad t\in\R\,,
\eeam
where $L^{(1)}$ and $L^{(2)}$ are two independent and identically distributed copies of $L$. 
Then $([L,L]_t^{(d)})_{t\in\R}$ is defined as the discrete part of the quadratic variation of the two-sided L\'evy process $L$.

Now we are able to define the  \emph{fractionally integrated} COGARCH$(1,1)$ process.

\begin{definition}[FICOGARCH$(1,d,1)$]\label{def:simaFicogarch}
Let $\alpha_0,\alpha_1,\beta_1>0$, $d\in(-0.5, 0)$ and $a>0$. 
Assume $L$ to be a L\'evy process with $\E(L^4_1)<\infty$. 
Let $(S^{a,d}_t)_{t\geq 0}$ be the fractional subordinator as in \eqref{eq:fracSvola}.
Define the stochastic process $(G_t)_{t\ge 0}$ by  
\beam\label{eq:PriceMvNficogarch}
  \mathrm{d}G_t = \sigma_t \mathrm{d}L_t\,, \qquad t\ge 0, 
\eeam
where the squared volatility $(\sigma^2_t)_{t\ge 0}$ is the solution of the SDE
\beam\label{eq:sdeVolaMvNficogarch}
  \mathrm{d}\sigma^2_t =-\beta_1 (\sigma^2_{t}-\alpha_0)\; \mathrm{d}t + \alpha_1 \sigma^2_{t} \; \mathrm{d}S^{a,d}_t\,,\quad t>0\,,
\eeam
with initial values $G_0$ and $\sigma_0^2$.
The model \eqref{eq:PriceMvNficogarch} with \eqref{eq:sdeVolaMvNficogarch} 
 is called \emph{fractionally integrated COGARCH}$(1,1)$ {\em process with fractional parameter} $d$ or {\em FICOGARCH}$(1,d,1)$.
 The stochastic volatility model \eqref{eq:sdeVolaMvNficogarch} is called \emph{FICOGARCH}$(1,d,1)$ {\em volatility process}.
\end{definition}

We can state the solution of the SDE \eqref{eq:sdeVolaMvNficogarch} explicitly. It can be verified by  integration by parts, which is applicable, since $(S_t^{a,d})_{t\ge0}$ has a.s. continuous and increasing sample paths, which we show in Proposition~\ref{prop4-modsub} (i) below.

\begin{proposition}
Consider the {FICOGARCH}$(1,d,1)$ volatility process as in \eqref{eq:sdeVolaMvNficogarch}.
Then for almost all sample paths the pathwise solution of the SDE \eqref{eq:sdeVolaMvNficogarch} with initial value $\sigma^2_0>0$ is given by 
\begin{align}\label{eq:solSdeVol}
	\sigma^2_t = e^{-X_t}\left(\sigma^2_0 + \alpha_0\beta_1 \int_0^t e^{X_s}\; \mathrm{d}s\right), \qquad t\ge 0,
\end{align}
with 
\beam\label{aux}
	X_t = \beta_1 t - \alpha_1 S^{a,d}_t, \qquad t\ge 0.
\eeam
\end{proposition}

In Figure~\ref{fig:ficogarch} we compare the sample paths of two FICOGARCH$(1,d,1)$ processes for different choices of $d\in(-0.5,0)$. 
The left column shows the price, return and volatility process of a FICOGARCH$(1,-0.01,1)$ process along with the sample autocorrelation function (acf) of the volatility process.
The right column depicts the same quantities of a FICOGARCH$(1,-0.40, 1)$ process driven by the same L\'evy process. 
One expects a slower decay of the acf of the volatility process for larger values of $d$, since the autocovariance function $\gamma_\Delta$ of the stationary  increments $\Delta S^{a,d}_t = S^{a,d}_{t} - S^{a,d}_{t-\Delta}$ satisfies
\begin{align*}
		\gamma_\Delta(h) \sim \var\big([L,L]_1\big) |d| a^d  \Delta^2 \; (h\Delta+a)^{d-1},\quad h\to\infty,
\end{align*} 
which is shown in Proposition~\ref{prop5-modsub} below.
There is no closed form expression for the autocovariance function of the volatility process. 
The approach used in the COGARCH model to derive the autovcovariance function is based on the independent increments property of the driving L\'evy process and hence won't work in this setting. 
But the expected behaviour of decay of the acf is confirmed by the two sample acfs in the bottom row of Figure~\ref{fig:ficogarch}.  
For simulating the fractional L\'evy process we approximated the process by the corresponding Riemann sums as explained in \citet[Section 2.4]{Marquardt-PhD}.

\begin{figure}[htbp]
\centering
\includegraphics[width=14.8cm]{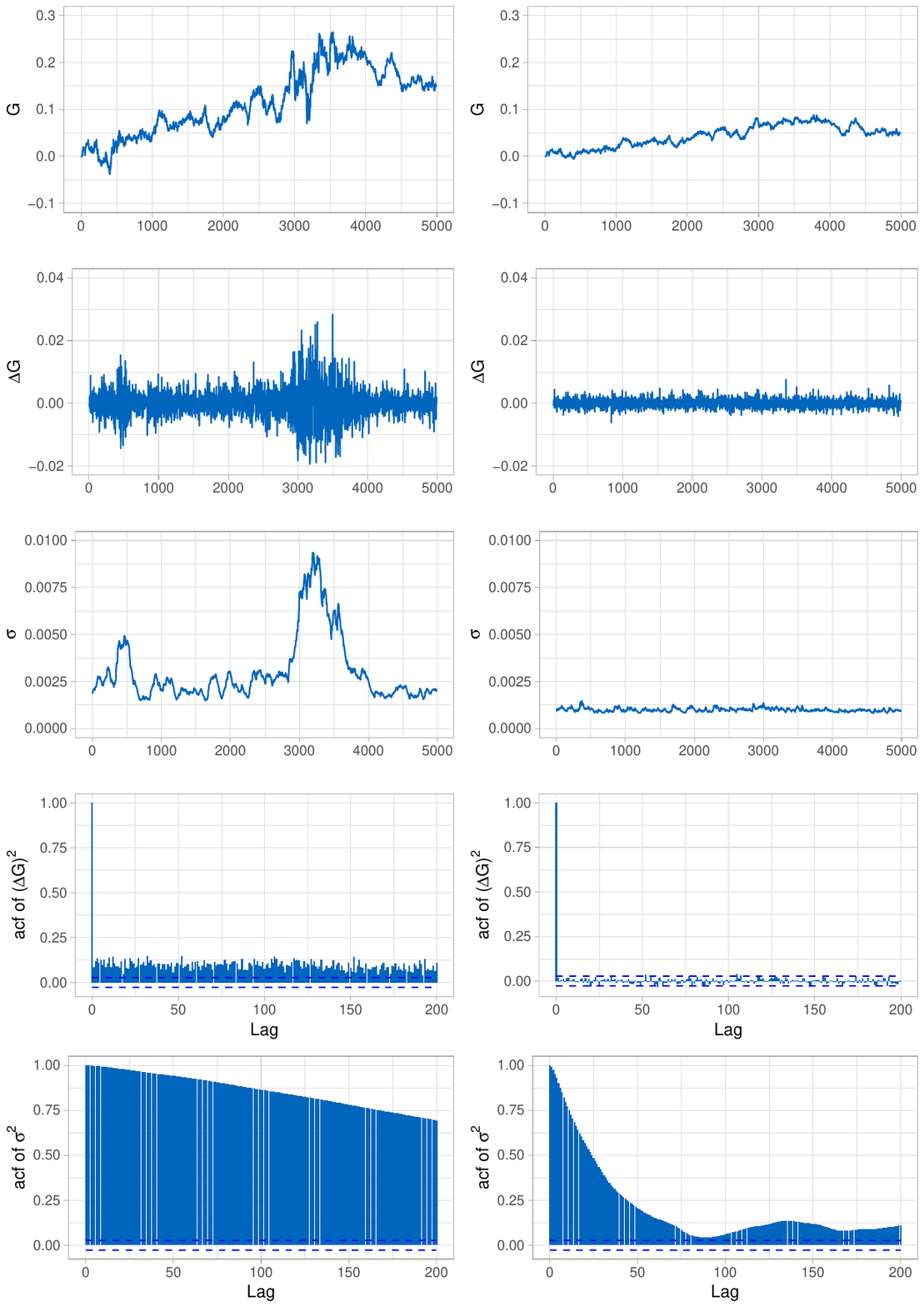}%
\caption{Simulation of the FICOGARCH$(1,d,1)$ process $G$ (top) with corresponding
return process $\Delta G$ (second row), volatility process $\sigma$ (third row) and sample acf $\widehat \gamma_\sigma$ of $\sigma$ (bottom) driven by a compound Poisson process $L$ with rate 5 and normally distributed jump sizes with mean zero and variance one half. 
The model parameters are $\alpha_0 = 0.0195, \alpha_1 = 0.0105, \beta_1 = 0.0513$ and $a=1.$ The fractional parameter $d$ is equal to $-0.01$ (left) and $-0.4$ (right). 
}
\label{fig:ficogarch}
\end{figure}

To derive a stationary version of the volatility process in our new model, we need some sample path properties of the fractional subordinator \eqref{eq:fracSvola}, which we present in a more general context in the next subsection.

\subsection{Sample path properties of fractional subordinators}

We first recall from \cite[Proposition~3.10]{ContTankov2003} that a subordinator $S:=(S_t)_{t\ge0}$ is a L\'evy process with characteristic triple $(\gamma_S,0,\nu_S)$, where $\gamma_S\ge0$ represents the drift of the process, it has no Gaussian component, and its L\'evy measure satisfies $\nu_S((-\infty,0])=0$, and as a consequence, it has almost surely increasing sample paths of finite variation.
In this section we will present some properties of a general  fractional subordinator 
\begin{align}\label{eq:fracSvolagen}
	S^{a,d}_t = \int_\R f_{a,d}(t,u)\; \mathrm{d}S_u, \qquad t\geq 0,
\end{align}
 where $f_{a,d}$ is the modified MvN kernel and $(S_t)_{t\in\R}$ is a two-sided subordinator defined from $S$ as in \eqref{2L}. 

The continuity of the sample paths of $(S^{a,d}_t)_{t\in\R}$ will follow from a representation of $S^{a,d}_t$ as improper Riemann integral. To derive such a representation, we will use the fact that  
  \begin{align*}
  S_t^{a,d} =  \begin{cases}
   \int_{-\infty}^t \Big((a+(-s))^d - (a+(t-s))^d \Big) \mathrm{d}S_s 
   + \int_t^0 \Big( (a+(-s))^d - a^d\Big) \mathrm{d}S_s & \mbox{for } t<0,\\
   \int_{-\infty}^0 \Big((a+(-s))^d - (a+(t-s))^d \Big) \mathrm{d}S_s 
   + \int_0^t \Big(a^d -(a+(t-s))^d \Big) \mathrm{d}S_s & \mbox{for } t\ge 0.
   \end{cases}
  \end{align*}
The following result is then an analogue of \citet*[Theorem~3.4]{Marquardt2006} for a fractional L\'evy process.

 \begin{proposition}\label{prop3-modsub}
Let $S$ be a  subordinator with $\E(S_1^2)<\infty$,  $d\in(-0.5,0)$ and $a>0$.
 Then $(S^{a,d}_t)_{t\in\R}$ as defined in \eqref{eq:fracSvolagen} has a modification that equals the improper Riemann integral
  \beam\label{eq:Riemann} 
  S_t^{a,d} = d \int_\R  \Big((a+(-s)_+)^{d-1} - (a+(t-s)_+)^{d-1}\Big)  S_s \mathrm{d}s - d a^{d-1}\int_0^t  S_s \mathrm{d}s,\quad t\in\R.
  \eeam
Moreover, \eqref{eq:Riemann} 
is continuous in $t$ and has sample paths of finite variation. 
\end{proposition}


Recall from \citet{ContTankov2003}, Corollary~3.1 and Proposition~3.10, that a subordinator $S$ with strictly increasing sample paths has representation $S_t=\gamma_S t +\sum_{0<u\le t}\Delta S_u$, for $t\ge0$; i.e. it has 
positive drift $\gamma_S>0$ and positive jumps of finite variation.
If this is the case, then the fractional subordinator $(S_t^{a,d} )_{t\ge0}$ also has strictly increasing sample path. 
This result is summarised in the following proposition. In addition we consider the a.s. limit behaviour of the two-sided process $(S^{a,d}_t)_{t\in\R}$.

\begin{proposition}\label{prop4-modsub}
Let $S$ be a subordinator and  $d\in(-0.5,0)$ and $a>0$.
 \begin{enumerate}[(i)]
    \item If $S$ has strictly increasing sample paths, then also $(S^{a,d}_t)_{t\ge0}$ has strictly increasing sample paths.
    \item The two-sided process $(S^{a,d}_t)_{t\in\R}$ has stationary increments.
    \item {If $\E(S_1)<\infty$,} then the following strong law of large numbers (SLLN) holds for the two-sided process $(S_t^{a,d} )_{t\in\R}$:
        \begin{align}\label{FLP24}
            \lim_{t\to \pm \infty}\frac1{|t|} \Big(S^{a,d}_t-\E(S_1)\int_\R f_{a,d}(t,u)\mathrm{d}u\Big)= 0\quad \mbox{a.s.}
        \end{align}
 \end{enumerate}
\end{proposition}

\subsection{Stationarity of the FICOGARCH\texorpdfstring{$(1,d,1)$}{(1,d,1)}}\label{s41}

To construct a stationary version of the FICOGARCH$(1,d,1)$ volatility model, we extend the squared volatility process of \eqref{eq:sdeVolaMvNficogarch} to the whole of $\R$ and use the SLLN from \eqref{FLP24}.
In this model the fractional subordinator \eqref{eq:fracSvolagen} is driven by $S=[L,L]^{(d)}$.
Then the auxilliary process \eqref{aux} is extended to
\beam\label{auxr}
  X_t = \beta_1 t-\alpha_1 S_t^{a,d}\,,\quad t\in\R\,.
\eeam
By Proposition~\ref{prop3-modsub} $(S^{a,d}_t)_{t\in\R}$ has a modification with finite variational and continuous sample paths, and hence $(X_t)_{t\in\R}$ has.

\begin{proposition}
\label{stat_limit}
{Let $S=[L,L]^{(d)}$} be the discrete part of the quadratic variation of a L\'evy process $L$. For $d\in(-0.5,0)$ and $a>0$ let $(S_t^{a,d})_{t\in\R}$  be the fractional subordinator as in \eqref{eq:fracSvolagen} and
 $(X_t)_{t\in\R}$ be the auxiliary process from \eqref{auxr} and let $\beta_1/\alpha_1 > a^d \E ([L,L]_1^{(d)})$. 
Then for $-\infty\le v<\infty$, the integral 
    \[
    \int_v^t e^{-(X_t-X_s)} \mathrm{d}s,\quad t>v,
    \]
exists a.s. as a Riemann integral.
\end{proposition}


\begin{lemma}
\label{stat_sigma_tilde}
Define
$$\widetilde\sigma^2_t := \int_{-\infty}^t e^{-(X_t-X_s)} \mathrm{d}s\,,\quad t\in\R\,.$$
Then for all $t_1<\cdots < t_m$, $m\in\N$, and $h\in\R$, we have
  \[
    (\widetilde\sigma^2_{t_1},\ldots,\widetilde\sigma^2_{t_m})\eqd (\widetilde\sigma^2_{t_1+h},\ldots,\widetilde\sigma^2_{t_m+h}).
    \]
\end{lemma}

\noindent
As a consequence of this general result, if in the solution \eqref{eq:solSdeVol}, we define $\sigma_0^2:=\alpha_0\beta_1\int_{-\infty}^0 e^{X_s} \mathrm{d}s$ , we get 
    \[
    \sigma_t^2 = \alpha_0\beta_1\int_{-\infty}^t e^{-(X_t-X_s)} \mathrm{d}s\,,\quad t\in\R,
    \]
which is a stationary process. The main theorem follows now from the preceding results
and the fact that $L$ has independent and stationary increments, which implies stationary increments of the price process $G$ as long as the volatility is a stationary process.

\begin{theorem}
\label{thm_stat_ficogarch} Let $L$ be a L\'evy process with $\E(L_1^4)<\infty$ and $([L,L]_t^{(d)})_{t\in\R}$ be the two-sided version of the discrete part of the quadratic variation process of $L$.
Let $d\in(-0.5,0)$ and $a>0$ and
assume that the parameters $\alpha_0,\alpha_1,\beta_1>0$ satisfy $\beta_1/\alpha_1 > a^d \E([L,L]_1^{(d)})$.
Let the squared FICOGARCH$(1,d,1)$ volatility process $(\sigma^2_t)_{t\geq 0}$ be given as in \eqref{eq:solSdeVol} with $\sigma_0^2 \eqd \alpha_0\beta_1\int_{-\infty}^0 e^{X_s} \mathrm{d}s$ independent of $L$.
Then $(\sigma^2_t)_{t\ge 0}$ is strictly stationary. 
Moreover, the process $(G_t)_{t\geq 0}$ as defined in \eqref{eq:PriceMvNficogarch} has stationary increments.
\end{theorem}

Our new approach extends immediately also to higher order COGARCH models.

\begin{remark}
The COGARCH($p,q$) process was introduced in \citet*{Brockwell2006}. By the same procedure as above we can generalise the FICOGARCH$(1,d,1)$ model in a straightforward way to its higher order analogue as follows.

Let $p$ and $q$ be integers such that $q\geq p\geq 1$. 
Further let $\alpha_0, \alpha_1,\dots, \alpha_p\in\R, \beta_1,\dots,\beta_q\in\R, \alpha_p\not=0,\beta_q\not=0$ and $\alpha_{p+1}=\cdots =\alpha_q=0$. Then we define the $q\times q$ matrix $\mathcal B$ and the vectors $\mathbf a$ and $\mathbf 1_q$ by
  \[
  {\mathcal B} = \begin{pmatrix}
  0 & 1 & 0 & \cdots & 0\\
  0 & 0 & 1 & \cdots & 0\\
  \vdots & \vdots & \vdots & \ddots & \vdots\\
  0 & 0 & 0 & \cdots & 1\\
  -\beta_q & -\beta_{q-1} & -\beta_{q-2} & \cdots & -\beta_1
  \end{pmatrix},\ 
  {\mathbf a} = \begin{pmatrix}
  \alpha_1\\ \alpha_2\\ \vdots\\ \alpha_{q-1}\\ \alpha_q
  \end{pmatrix},\
  {\mathbf 1_q} = \begin{pmatrix}
  0\\ 0\\ \vdots\\ 0\\ 1
  \end{pmatrix},
    \]
with $\mathcal{B}:= -\beta_1$ if $q=1$. 
Then for a L\'evy proceess $L$ with $\E(L_1^4)<\infty$, we define the squared volatility process $(\sigma^2_t)_{t\geq 0}$ with parameters $\mathcal{B}, {\mathbf a}, \alpha_0$ by
 \[
 \sigma^2_t = \alpha_0 + {\mathbf a}^\top {\mathbf Y_t},\qquad t\geq 0,
 \]
where the state process $({\mathbf Y}_t)_{t\geq 0}$ is the unique solution of the SDE
  \[
  \mathrm{d}{\mathbf Y}_t = \mathcal{B}{\mathbf Y}_t \mathrm{d}t + {\mathbf 1_q}(\alpha_0 + {\mathbf a}^\top {\mathbf Y}_t) \mathrm{d}S_t^{a,d},\qquad t>0,
  \]
with initial value ${\mathbf Y}_0$, independent of $L$.
Furthermore, $(S^{a,d}_t)_{t\geq 0}$ is the fractional subordinator as defined in equation \eqref{eq:fracSvola}. 
If the process $(\sigma^2_t)_{t\geq 0}$ is strictly stationary and almost surely non-negative, we define the FICOGARCH($p,d,q$) process $(G_t)_{t\geq 0}$ with parameters $\mathcal{B}, {\mathbf a}, \alpha_0$ and some initial value $G_0$ as the solution of the SDE
  \[
  \mathrm{d}G_t = \sigma_{t}\mathrm{d}L_t, \qquad t>0.
  \]
The volatility process of the COGARCH$(1,1)$ and also of the FICOGARCH$(1,d,1)$ is non-negative by definition. 
This is not necessarily the case for the COGARCH$(p,q)$ model for $q\ge p>1$. 
Therefore, conditions as formulated in \citet[Theorem 5.1]{Brockwell2006} have to be considered to assure non-negativity of the volatility process.
\end{remark}

\section{An application example}
\label{s3}

\subsection{Parameter estimation}\label{subsec:param_est}

Estimation of the FICOGARCH model is not as straightforward as e.g. for the COGARCH(1,1) model. There the
dependence structure of the squared returns $(\Delta G_{\Delta i})^2 := (G_{\Delta i} - G_{\Delta (i-1)})^2$ is explicitly
known and has been used to define a method of moment estimator (\cite{Haugetal2007}), a prediction based estimator (\citet*{BibbonaNegri2015}), an indirect inference estimator (\citet*{Thiago}) or a pseudo-maximum-likelihood estimator (\citet*{Malleretal2008}).  

Fortunately we are able to simulate the FICOGARCH model. This allows us to apply a simulation based version of the generalised method of moments (GMM), due to \cite{mcfadden1989}. The idea is to compute  
moments for the observed return data, and match them with empirical moments computed from simulated data of the FICOGARCH model. The method is described in detail in the algorithm below. 

We illustrate the estimation for equally spaced observations of the price process $G_{\Delta i}, i = 0,\ldots,n$, giving return data
$\Delta G_{\Delta i} = G_{\Delta i} - G_{\Delta (i-1)},$ $i = 1,\ldots,n$, with true parameter $\theta^0 = (\alpha_0^0,\alpha_1^0,\beta_1^0,d^0) \in \Theta$, where $\Theta = (0,\infty)^3 \times (-0.5,0)$ denotes the parameter space. \\

\noindent \textbf{Algorithm}

\begin{enumerate}

\item[(1)] Compute the moment estimator 
\[
\widehat{\mu}_n := \frac{1}{n} \sum_{i=1}^n (\Delta G_{\Delta i})^2
\]
and for fixed $s \geq 3$ the empirical autocovariances $\widehat{\gamma}_n:= (\widehat{\gamma}_n(0),\widehat{\gamma}_n(1),\ldots,\widehat{\gamma}_n(s))$  as

\[
    \widehat{\gamma}_n(h) 
        := \frac{1}{n}\sum_{i=1}^{n-h} \bigg( (\Delta G_{\Delta(i+h)})^2 - \widehat{\mu}_n\bigg)
            \bigg( (\Delta G_{\Delta i})^2 - \widehat{\mu}_n\bigg), \quad h = 0,\ldots,s.
\]

\item[(2)] Compute the empirical autocorrelations 
\[
\widehat{\rho}_n =(\widehat{\rho}_n(1),\dots,\widehat{\rho}_n(s)):= (\widehat{\gamma}_n(1)/\widehat{\gamma}_n(0),\ldots,\widehat{\gamma}_n(s)/\widehat{\gamma}_n(0)).
\]

\item[(3)] For $\theta = (\alpha_0,\alpha_1,\beta_1, d) \in \Theta $ simulate a sample path $(\Delta G_{\Delta i}(\theta))_{i =0,\ldots,n}$ of return data. \\
Repeat steps (1) and (2) to obtain $\widehat{\mu}_{n,\theta}$ and $\widehat{\rho}_{n,\theta}$.

\item[(4)] Define the function $\mathcal{L}^{(s)}(\theta)$ by
\begin{equation}\label{eq:score_iim}
\mathcal{L}^{(s)}(\theta) =  \sum_{h=1}^s ( \widehat{\rho}_{n,\theta}(h) - \widehat{\rho}_n(h))^2 + (\widehat{\mu}_{n,\theta} - \widehat{\mu}_n)^2
\end{equation}
The simulation-based GMM estimator $\widehat{\theta}_n$ of $\theta^0$ is given by:
\begin{equation*}
\widehat{\theta}_n = \mathrm{argmin}_{\theta \in \Theta} \mathcal{L}^{(s)}(\theta).
\end{equation*}

\end{enumerate}

Before commenting on some practical issues of the above algorithm we present a useful simplification.

\begin{lemma}\label{remark:scale_property_ficogarch}
Let $\Delta G_t = G_t - G_{t-\Delta}$ be an increment of a stationary FICOGARCH$(1,d,1)$ process with parameter $(\alpha_0,\alpha_1,\beta_1,d)$ satisfying the conditions of Theorem \ref{thm_stat_ficogarch}. Then, for any $k>0$, the scaled increment $k\cdot \Delta G_t$ has the same distribution as the increment of a FICOGARCH$(1,d,1)$ with parameter $(k^2\alpha_0,\alpha_1,\beta_1,d)$.
\end{lemma}

This result follows from the fact that the stationary squared volatility process of the scaled FICOGARCH$(1,d,1)$ process $(k G_t)_{t\geq 0}$ is given by

\[
    (k\sigma_{t})^2 = k^2\alpha_0 \beta_1 \int_{-\infty}^t e^{-(X_t-X_s)}\mathrm{d}s,\qquad t\in\R.
\]
Since $(X_t)_{t\in\R}$ is independent of $\alpha_0$, it follows that
    \[
    \Delta G(\alpha_0,\alpha_1,\beta_1, d) \overset{d}{=} \sqrt{\alpha_0}\Delta G(1,\alpha_1,\beta_1, d)\,.
    \]
Hence, 
$\rho_{[\Delta G(\alpha_0,\alpha_1,\beta_1, d)]^2}(h) = \rho_{[\Delta G(1,\alpha_1,\beta_1, d)]^2}(h)$ for every $h \geq 0$,
where $\rho_{[\Delta G(\theta)]^2}$ denotes the autocorrelation function of squared increments
of a stationary FICOGARCH$(1,d,1)$ process with parameter $\theta$.

The minimisation of the score function (\ref{eq:score_iim}) is therefore performed in two steps. First we keep $\alpha_0 = 1$ fixed and minimise the first term in \eqref{eq:score_iim} with respect to $(\alpha_1,\beta_1,d)$.
Second, since we have $\widehat{\mu}_{n,(\alpha_0,\alpha_1,\beta_1,d)} = \alpha_0\widehat{\mu}_{n,(1,\alpha_1,\beta_1,d)}$
by Lemma \ref{remark:scale_property_ficogarch}, the second term in \eqref{eq:score_iim} is minimised by choosing: 
    \[
    \widehat{\alpha}_0 = \frac{\widehat{\mu}_n}{\widehat{\mu}_{n,(1,\widehat{\alpha}_1,\widehat{\beta}_1,\widehat{d})}}\,.
    \]

We now proceed to investigate the small sample behaviour of the proposed estimator. We simulate 5\,000 equidistant observations of $\Delta G$ with $\Delta = 1$. The driving L\' evy process $L$ is compound Poisson with rate one and standard normally distributed jumps. The cut-off for the computation of the empirical acf has been chosen as $s = 80$ lags. Table~\ref{table:iim_ficogarch} summarises  the estimation results for $N=100$ simulation runs. The five examples shown in the table only differ in $d$, which varies between -0.01 and -0.49 as indicated in the right column of the table. The estimation of the parameters $\beta_1$ and $d$ becomes less efficient as $d$ decreases towards $-0.49$. This is due to the fact that, when $d$ decreases, the model becomes less dependent so that the acf of the squared returns is less informative about the model parameters. The estimation results for the other two parameters are not affected by the decrease of $d$ and are overall satisfying.\\

\begin{table}
\centering
\begin{tabular}{lllll}
  \hline
&$\alpha_0$& $\alpha_1$ & $\beta_1$ & $d$\\
  \hline
TrueValue&0.04000&0.08000&0.34000&-0.01000\\
Mean&0.05093 (0.01932)&0.09080 (0.03621) &0.35720 (0.16520) &-0.14400 (0.14914)\\
RMSE&0.00049&0.00143&0.02759&0.04020\\
  \hline
TrueValue&0.04000&0.08000&0.34000&-0.10000\\
Mean & 0.04218 (0.02751) & 0.08920 (0.03045) & 0.31720 (0.12431) &-0.12400 (0.13566)\\
RMSE&0.00076&0.00101&0.01597&0.01898\\
  \hline
TrueValue&0.04000&0.08000&0.34000&-0.25000\\
Mean&0.04798 (0.01267) &0.13640 (0.07894) &0.57480 (0.31452) &-0.29400 (0.16169)\\
RMSE&0.00022&0.00941&0.15405&0.02808\\
  \hline
TrueValue&0.04000&0.08000&0.34000&-0.40000\\
Mean&0.05028 (0.01425)&0.10920 (0.09130) &0.72360 (0.29278) & -0.29400 (0.16415)\\
RMSE&0.00031&0.00919&0.23287&0.03818\\
  \hline
TrueValue&0.04000&0.08000&0.34000&-0.49000\\
Mean&0.04802 (0.01293) &0.09800 (0.09364) &0.73640 (0.30812)&-0.27400 (0.17523)\\
RMSE&0.00023&0.00909&0.25207&0.07736\\
   \hline
\end{tabular}
\caption{Empirical mean, standard deviation (in brackets)  and root mean squared error (RMSE) of the estimated FICOGARCH$(1,d,1)$ parameters.}\label{table:iim_ficogarch}
\end{table}

\subsection{The FICOGARCH model fitted to FOREX data }

The FICOGARCH$(1,d,1)$ model has been fitted to log-returns of two different exchange rate data, namely EUR (Euro)/USD (United States Dollar) and GBP (British Pound)/USD. The data set have a total length of 4\,436 (from December 10, 1999 to December 12, 2016) and 10\,436 (from December 10, 1976 to December 12, 2016) data points, respectively.

\begin{figure}[h]
\centering
  \includegraphics[width=14.5cm]{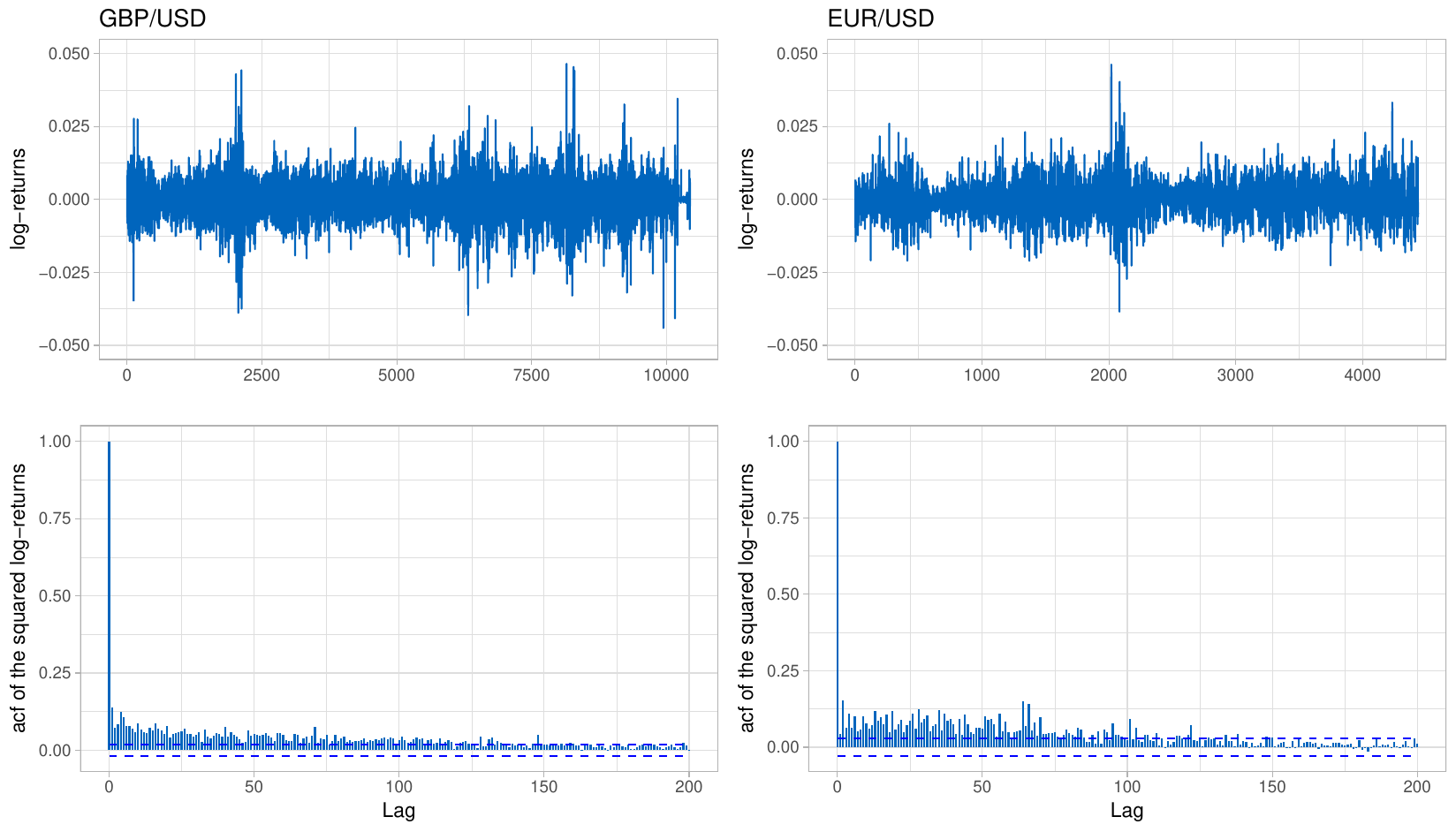} 
\caption{Log-returns of foreign exchange rates (\textit{top}) and sample acf of
squared log-returns (\textit{bottom}) for GBP/USD (\textit{left}) and USD/EUR 
(\textit{right}) exchange rates.}
\end{figure}

The sample autocorrelation function of the squared log-returns indicates some moderately long memory behaviour as it is described by the FICOGARCH model; cf. the corresponding plots in Figure~\ref{fig:ficogarch}. 

We applied the Algorithm of Section \ref{subsec:param_est} to each data set to estimate the parameters $\alpha_0,\alpha_1,\beta_1$ and $d$ of the FICOGARCH(1,d,1) model. The parameter $a$ of the fractional process $(S^{a,d}_t)_{t\in\R}$ was set equal to 1 in both cases. 
For the simulation part the driving L\'evy process $L$ was chosen to be a compound Poisson process with rate one and standard normally distributed jumps. The estimated parameters are shown in Table \ref{table:forex_ficogarch}. The last column represents the realised minimum of the score function $\mathcal{L}^{(s)}(\hat\theta)$. 

\begin{table}[h]
\centering
\begin{tabular}{rrrrrr}
  \hline
  \hline\\[-3mm]
&$\widehat{\alpha}_0$&$\widehat{\alpha}_1$&$\widehat{\beta}_1$&$\widehat{d}$ & $\mathcal{L}^{(s)}(\hat\theta)$ \\
\hline
GBP/USD & 1.7656e-6  &  0.0274  & 0.1099  & -0.0300 & 0.0830 \\
EUR/USD & 1.2295e-5  &  0.0390  & 0.1460  & -0.0150 & 0.1607 \\
   \hline
   \hline
\end{tabular}
\caption{Parameter estimates $\widehat{\alpha}_0$, $\widehat{\alpha}_1$, $\widehat{\beta}_1$, $\widehat{d}$ for GBP/USD and EUR/USD exchange rate data.}\label{table:forex_ficogarch}
\end{table}

The estimated values of $d$ are  rather close to 0 for both series, which indicates that the fitted FICOGARCH model has an acf with a rather slow decay. The estimated models are stationary, since the ratios $0.1099/0.0274$ and $0.1460/0.0390$ are both larger than $1 = a^{\widehat d}\E([L,L]_1)$. The estimates for $\alpha_0$ are both close to zero, which is due to the scale of the data (cf. Lemma~\ref{remark:scale_property_ficogarch}).

\section{Fractional subordinators in stochastic volatility models}\label{s4}

In this section we discuss the fractional subordinator as defined in \eqref{eq:fracSvolagen} with its specification \eqref{eq:fracSvola} as driving process of the FICOGARCH$(1,d,1)$ squared volatility.
The non-negativity needed for the squared volatility can be achieved in various ways, one possibility consists in the integral transform of a subordinator with a positive kernel as the  Ornstein-Uhlenbeck type volatility model of \citet{Barndorff-Nielsen2001}.
Long range dependence is then introduced by replacing the OU kernel by a fractional kernel as suggested for instance in \cite{anh2002}, or by taking the integral transform of a fractional subordinator with an appropriate kernel as in \citet*{Brockwell2005}, Section~8.
Also time-changed subordinators as in \citet*{Carr2003} can be modified to introduce long range dependence.
 \citet{Bender2009} replaces classical time-change models by the integral transform of a subordinator with a MG-kernel as in \eqref{mgkernel}  which, when used as a random time change process, also results in a long range dependence model.  



As explained in Section~1.2, the classical MG-kernel in \eqref{mgkernel} and MvN-kernel in \eqref{mvnkernel} have certain drawbacks, which make them inappropriate for fractional subordinator modelling for long range dependence in stochastic volatility.
The MG-fractional subordinator $(\tilde{L}_t^d)_{t\geq 0}$ defined in \eqref{fracsub} does not necessarily lead to stationarity of the squared volatility process.
The MvN-fractional subordinator $(L^d_t)_{t\in\R}$ defined in \eqref{fracsub}
is  well-defined for any subordinator $S$, if $d\in (-0.5,0)$.
But due to the singularity at $s=t$ 
$(S^d_t)_{t\in\R}$ has discontinuous and unbounded sample paths with positive 
probability (cf. \citet[Theorem 4]{Rosinski1989}).
To overcome these drawbacks, we bound the MvN-kernel at its singularities. 

Observe that for every $t\in\R$ the MvN-kernel $\mvn(t,\cdot)$ is up to a constant given by the function $s\mapsto g_d(t-s)-g_d(-s)$, where $g_d$ is defined by $g_d(x):=x^d_+$. 
This suggests to bound $g_d$ at its singularities $s=0$ and $s=t$ by incorporating a shift $a>0$, leading to
the following definition of a modification of the Mandelbrot-van-Ness kernel.

\pagebreak

\begin{definition}
Let $d\in (-0.5,0)$ and $a>0$. For each $t\in\R$ the \emph{(non-normalised) modified MvN-kernel} is given by
\begin{align}\label{eq:modMvnKernel}
\begin{aligned}
f_{a,d}(t,s) &= g_{a,d}(-s) - g_{a,d}(t-s) \\
			& := \left(a+(-s)_+\right)^d - \left(a+(t-s)_+\right)^d, \qquad s\in\R.
\end{aligned}
\end{align}
\end{definition}

\begin{figure}[htbp]
\begin{center}
\includegraphics[width=16cm,height=7cm]{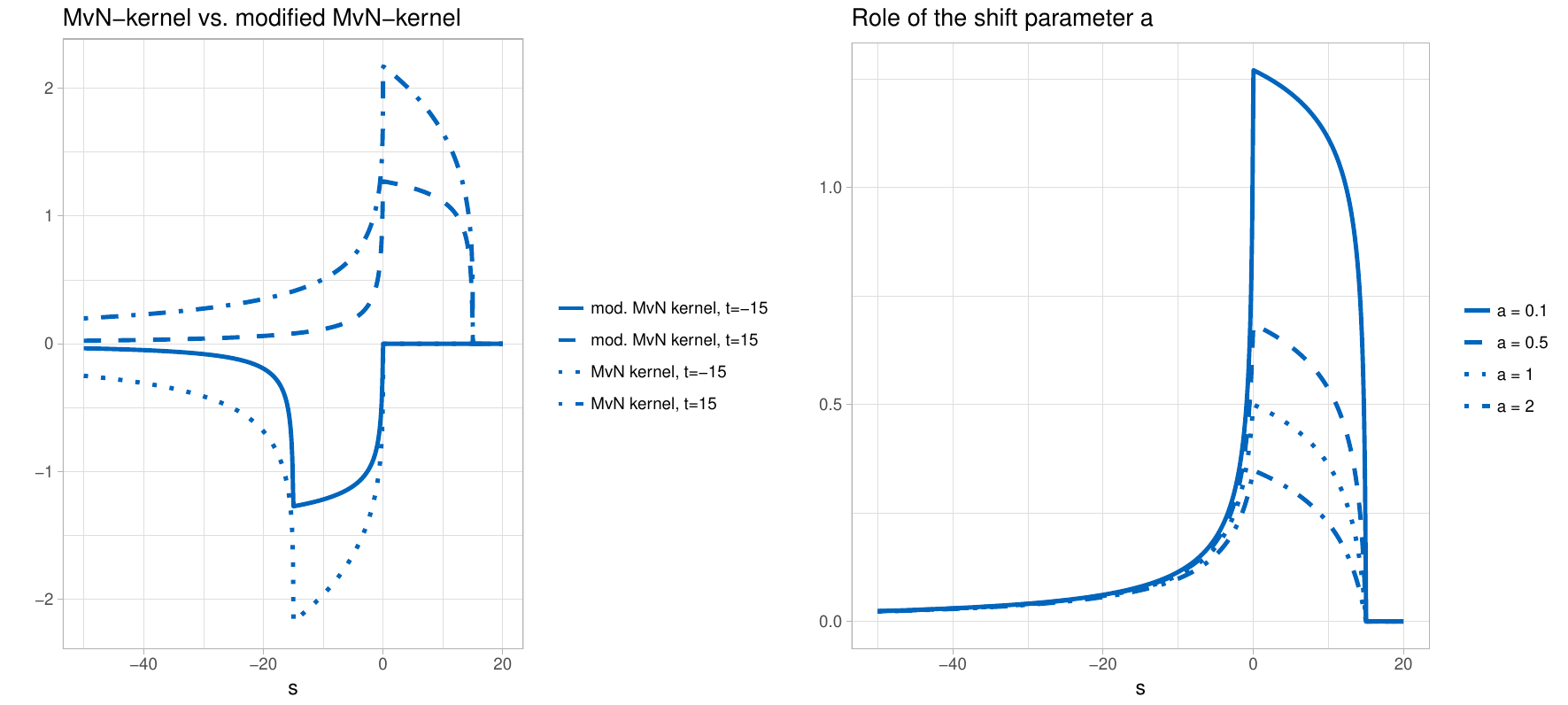}
\end{center}
\caption{Left: Comparison of the MvN kernel $\mvn(t,\cdot)$ and the modified MvN kernel $f_{a,-d}(t,\cdot)$ for $d=0.25, a=0.1$. Right: Modified MvN kernel $f_{a,d}(t,\cdot)$ with $d=-0.25$ for different values of the shift parameter $a$ for $t=15$.}
\end{figure}

Some properties of the new kernel are summarised in the next proposition. 
\begin{proposition}
\label{prop1-modsub}
For  $d\in (-0.5,0)$ and $a>0$ consider the modified MvN-kernel $f_{a,d}$ as in \eqref{eq:modMvnKernel}. Then the following holds.
\begin{enumerate}[(i)]
    \item $f_{a,d}(t,\cdot) \geq 0$ for all $t\ge 0$,
	\item $f_{a,d}(t,\cdot)$ is continuous for all $t\in\R$,
	\item $|f_{a,d}(t,s)| \le a^d$ for all $t,s\in\R$,
 \item $|f_{a,d}(t,s)| \sim |s|^{d-1} |dt|$\  as $s\to-\infty$  (equivalently, $\lim\limits_{s\to -\infty}{|f_{a,d}(t,s)|}/{(|s|^{d-1}|dt|)} = 1$) for all $t\in\R$.
 \item $f_{a,d}(t,\cdot)\in L^\delta (\R)$ for all $\delta>1/|d-1|$;
in particular, $f_{a,d}(t,\cdot)\in L^1(\R)\cap L^2(\R)$ for all $t\in\R$.
\end{enumerate}
\end{proposition}

The new kernel can now be used to define a fractional subordinator analogously to \eqref{fracsub}. The existence
of such an integral is shown in the following proposition.

\begin{proposition}
\label{prop2-modsub}
Let $a>0$ and $d\in (-0.5,0)$, and let $S$ be a subordinator satisfying $\E(S_1^2)<\infty$. Then
the fractional subordinator
\begin{align}
\label{frac_sub_modmvn}
	S^{a,d}_t = \int_\R f_{a,d}(t,u)\; \mathrm{d}S_u
\end{align}
exists for all $t\in\R$ as limit in the $L^2(\Omega)$-sense and hence in probability.

\end{proposition}

\pagebreak

\begin{remark}
Some related kernels have been used in the literature, and we want to comment on some of them.
\begin{itemize}
\item[(a)] For comparison, recall from above that the MvN kernel is for $d\in(-0.5,0)$ and $t\in\R$ defined as
 \[
 \mvn(t,s)=\frac{1}{\Gamma(1+d)} (g_d(t-s)-g_d(-s))= \frac{-1}{\Gamma(1+d)} \frac{\mathrm{d}}{\mathrm{d}s}\int_s^\infty \mathds1_{(0,t)} (v) (v-s)^{d} \mathrm{d}v,\quad t,s\in\R,
 \]
see \citet[Lemma 1.1.3]{Mishura2008}. 
As it is in $L^2(\R)$, this kernel can generate fBm and fLm on $\R$ for symmetric driving processes.
\item[(b)] The approach in \citet*{Brockwell2005} to construct a FICARMA process driven by a subordinator suggests to use the kernel
  \[
   \frac{-1}{\Gamma(1+d)} \frac{\mathrm{d}}{\mathrm{d}s}\int_s^\infty \mathds1_{(0,t)} (v) \min (a^d, (v-s)^{d}) \mathrm{d}v,\quad t,s\in\R,
  \]
for $d\in(-0.5,0)$ and some $a>0$. For us it was computationally advantageous to bound $g_d$ and not the integrand in the above representation. 
\item[(c)] \citet*{meerschaertSPA} suggest a tempered fractional kernel of the form
  \[
f_d^{\rm temp}(t,s) :=  e^{-\lambda (t-s)_+} (t-s)_+^d - e^{-\lambda (-s)_+} (-s)_+^d ,\quad t,s\in\R,
  \]
for {$d > -0.5$} and $\lambda>0$, which also has an integral representation using tempered fractional integrals, see \citet[Definition 2.1]{meerschaertSPA} for details. For every $t\in\R$ the kernel $f_d^{\rm temp}(t,\cdot)$ belongs to  $L^p(\R)$ for all $p\geq 1$. 
The autocovariance function of the increments of the corresponding tempered fractional process decreases for small lags algebraically, but for large lags exponentially.  It is therefore called a semi-long range dependence model.
\item[(d)] \citet*{KluMatsui} generalise fLp's  by allowing the kernel to be regularly varying, which results in functional central limit theorems for scaled Ornstein-Uhlenbeck processes driven by such generalized fLp's. 
\end{itemize}
\end{remark}


\subsection{Cumulant generating function and moments}

Let $S$ be a subordinator without drift, then its cumulant generating function is given by
$\ln \E(e^{h S_1})= \int_0^\infty (e^{h z}-1)\nu_S(\mathrm{d}z)$, where $\nu_S$ is the L\'evy measure of $S$ (see Chapter~4.2.2 in \citet{ContTankov2003}). 
\citet*[Proposition~2.6]{Rajput1989} {expresses} the cumulant generating function of  $(S^{a,d}_t)_{t\in\R}$ as
  \[
  \ln \E(e^{h S_t^{a,d}}) = \ln \E\Big( \exp\Big\{ h \int_\R f_{a,d}(t,u)\; \mathrm{d}S_u\Big\}\Big) 
  = \int_\R \int_0^\infty (e^{h f_{a,d}(t,u)z}-1) \nu_S(\mathrm{d}z) \mathrm{d}u.
  \]
The $k$-th cumulant of $S_t^{a,d}$ is then given by
$\kappa^k(S_t^{a,d})=\frac{d^k}{dh^k} \ln \E(e^{h S_t^{a,d}})\big|_{h =0}$. For the $k$-th derivative in 0 we obtain, provided the corresponding L\'evy cumulant exists,
  \[
  \kappa^k (S_t^{a,d}) = \int_0^\infty  z^k\nu_S(\mathrm{d}z) \int_\R f_{a,d}^k(t,u) \mathrm{d}u = \kappa^k (S_1) \int_\R f_{a,d}^k(t,u) \mathrm{d}u.
  \]
where the integral exists by Proposition~\ref{prop1-modsub} (v) for all $k>1$.
From this we calculate the mean and variance as 
  \begin{align*}
  \E(S_t^{a,d}) &= \E(S_1) \int_\R f_{a,d}(t,u) \mathrm{d}u,\\
\var(S_t^{a,d})  &= \var(S_1) \int_\R f_{a,d}^2(t,u) \mathrm{d}u.
  \end{align*}

\subsection{Properties of the increments of the fractional subordinator}

Next we ask whether $f_{a,d}$ serves its purpose in the sense that the increments of $(S^{a,d}_t)_{t\in\R}$ exhibit a long range dependence structure. 
We start with an auxiliary result.

\begin{lemma}\label{lem:fquad}
Let $d\in(-0.5,0)$ and $a>0$. 
Then the modified MvN-kernel $f_{a,d}(t,\cdot)$ as in \eqref{eq:modMvnKernel} satisfies for $t > 0$
	\[
	\int_\R f_{a,d}^2(t,u)\; \mathrm{d}u 
	= C + a^{2d}\; t - \frac{2a^d}{d+1} (t+a)^{d+1} + \frac{1}{2d+1} (t+a)^{2d+1} + c(t)  t^{2d+1} 
	\]
with 
 \beam\label{eq:fc}
 C = a^{2d+1} \Big(\dfrac{2}{d+1} - \dfrac{1}{2d+1} \Big)\quad\mbox{and}\quad
 c(t) = \int_{-\infty}^{-a/t} \left[(1-y)^d - (- y)^d\right]^2\; dy,
 \eeam
 and
\beam\label{eq:asymptCt}
\lim_{\tto} c(t)  = \frac{\Gamma(d+1)}{\Gamma(2d+2)\sin(\pi(d+0.5))}
   +\frac{1}{2d+1}.
   \eeam 
\end{lemma}

From this we are able to derive the asymptotic behaviour
of the autocovariance function of the increments of $(S^{a,d}_t)_{t\in\R}$. 

\begin{proposition}\label{prop5-modsub} 
Let $d\in(-0.5,0)$ and $a>0$, and let $S$ be a subordinator satisfying $\E(S_1^2)<\infty$.
Let $(S^{a,d}_t)_{t\in\R}$ be as defined in \eqref{frac_sub_modmvn}, then the following holds. 
    Let $\Delta>0$ be fixed and $s+\Delta\le t$ such that $t-s=h\Delta$ for some $h>0$.
    Then the two increments $S^{a,d}_{t+\Delta}-S^{a,d}_t$ and $S^{a,d}_{s+\Delta}-S^{a,d}_s$ of length $\Delta$ have covariance
		\[ 
		\gamma_\Delta(h) := \cov\left(S^{a,d}_{s+(h+1)\Delta}-S^{a,d}_{s+h\Delta}, S^{a,d}_{s+\Delta} - S^{a,d}_{s}\right),
		\]
    which satisfies 
		\begin{align*}
		\gamma_\Delta(h) \sim \var\big(S_1\big) |d| a^d  \Delta^2 \; (h\Delta+a)^{d-1},\quad h\to\infty.
		\end{align*} 
\end{proposition} 

 The increments of $(S^{a,d}_t)_{t\in\R}$ do not have long memory in the standard sense, since the autocovariance function 
is integrable. However, it decreases algebraically, and for $d$ close to zero we approximately obtain long memory. 
This is in analogy to the asymptotic rate of decay of the modified CARMA kernel $g_{a,d}$ in the case of subordinator-driven CARMA processes as in \citet*[Section~8]{Brockwell2005}.

\section{Conclusion and outlook}

We have introduced a new FICOGARCH$(1,d,1)$ model, which is strictly stationary and exhibits an algebraic decay in its autocovariances. Moreover, we have shown how to extend this model to a FICOGARCH\linebreak $(p,d,q)$
process of arbitrary order. The properties of the model present it as an appropriate model for high-frequency financial data exhibiting moderate long memory.
We have also presented a simple estimation method for the  model parameters including the fractional  parameter, which works well in a simulation study. 
We have also applied this method to exchange rate data. 
More sophisticated estimation procedures are envisaged like the estimation of $d$ in a preliminary step by an appropriate estimator. Then in a second step the FICOGARCH parameters can be estimated by modified COGARCH estimators as suggested in  \citet{BibbonaNegri2015}, \cite{Haugetal2007}, \citet{Malleretal2008}, or \citet{Thiago}.
This is a version of the semiparametric approach suggested in \citet*{Robinson1994}. 
Alternatively, a quasi-maximum likelihood estimation as in \citet*{TsaiChan2005} can be applied to estimate $d$ and all model parameters in one go.



\appendix

\section{Appendix}


\subsection*{Proofs of Section \ref{s2}}

\begin{proof}[Proof of Proposition \ref{prop3-modsub}]
First note that for every finite variational sample path of $(S^{a,d}_t)_{t\in\R}$  we can apply partial integration and obtain for $t>0$ 
 \begin{align}\label{Mfinite}
  S_t^{a,d} 
   &= \lim_{u\to -\infty} \int_{u}^0 \Big((a+(-s))^d - (a+(t-s))^d \Big) \mathrm{d}S_s 
   + \int_0^t \Big(a^d -(a+(t-s))^d \Big) \mathrm{d}S_s \nonumber\\
   & =  - (a^d -(a+t)^d) S_0 - d \int_0^t (a+(t-s))^{d-1} S_s \mathrm{d}s + (a^d -(a+t)^d) S_0 \nonumber\\
   & \quad - \lim_{u\to -\infty}  \Big((a+(-u))^d - (a+(t-u))^d \Big)  S_u \nonumber\\ 
   & \quad + \lim_{u\to -\infty}  d \int_u^0  \Big((a+(-s))^{d-1} - (a+(t-s))^{d-1} \Big)S_s  \mathrm{d}s
 \end{align}
Now recall that by the SLLN (cf. \citet[Proposition 36.3]{Sato1999}) that
  \[
 \lim_{s\to-\infty}\frac{ S_s}{|s|}\rightarrow \E(S_1)\quad a.s.
  \] 
Therefore, we get
  \begin{align*}
   0  &\leq   \lim_{u\to -\infty}  \Big((a+(-u))^d - (a+(t-u))^d \Big)  S_u 
	 \leq  \lim_{u\to -\infty} d t \, (a-u)^{d-1} \, S_u\\  
	& =  \lim_{u\to -\infty} d t \frac{S_u}{|u|} \, \Big(1- \frac{a}{u}\Big)^{-1} \, (a-u)^d = 0\,,
  \end{align*}
 and, thus, the  limit integral in \eqref{Mfinite} is finite.  
This yields 
  \[
  S_t^{a,d} =  d \int_{-\infty}^0  \Big((a-s)^{d-1} - (a+(t-s))^{d-1}
  \Big)S_s  \mathrm{d}s - d \int_0^t (a+(t-s))^{d-1} S_s \mathrm{d}s.
  \]
For $t\le 0$ we get {analogously}
 \[
  S_t^{a,d} =  d \int_{-\infty}^t  \Big((a-s)^{d-1} - (a+(t-s))^{d-1}
  \Big)S_s  \mathrm{d}s + d \int_t^0 (a-s)^{d-1} S_s \mathrm{d}s.
  \]
Continuity follows from dominated convergence {from Proposition~\ref{prop1-modsub} (iii). }
Since the kernel $f_{a,d}(t,\cdot)$  and almost all paths of $(S_t)_{t\in\R}$ have finite variation, it follows that also $(S^{a,d}_t)_{t\in\R}$ has sample paths of finite variation.
\end{proof}

\begin{proof}[Proof of Proposition \ref{prop4-modsub}]
(i)\  For $t\geq 0$ we have the representation
  \[
  S_t^{a,d} =  \int_{-\infty}^0 \Big((a+(-s))^d - (a+(t-s))^d \Big) \mathrm{d}S_s 
   + \int_0^t \Big(a^d -(a+(t-s))^d \Big) \mathrm{d}S_s\,.
  \]
The functions $t\mapsto \big((a+(-s))^d - (a+(t-s))^d\big) \mathds{1}_{(-\infty, 0)}(s)$ and $t\mapsto  \big(a^d -(a+(t-s))^d\big) \mathds{1}_{(0,t)}(s)$ are for each fixed $s$ positive. 
Therefore, it follows from the above representation that $(S^{a,d}_t)_{t\in\R}$ has strictly increasing sample path if $(S_t)_{t\in\R}$ has strictly increasing sample path.

(ii) For $n\in\N$, $t_0<t_1<\cdots<t_n$ and $a_1,\ldots,a_n\in\R$ we use the Cram\'er-Wold device and calculate
\beao
\sum_{i=1}^n a_i (S^{a,d}_{t_i+h}-S^{a,d}_{t_{i-1}+h})
&=& \sum_{i=1}^n a_i 
\int_\R (g_{a,d}(t_{i-1}+h-u) - g_{a,d}(t_{i}+h-u)) \mathrm{d}S_u\\
&\eqd & \sum_{i=1}^n a_i \int_\R (g_{a,d}(t_{i-1}-v) - g_{a,d}(t_{i}-v)) \mathrm{d}S_v\\
&=& \sum_{i=1}^n a_i  (S^{a,d}_{t_i}-S^{a,d}_{t_{i-1}}),
\eeao
where we have used the stationary increments of $(S_t)_{t\in\R}$.
 
(iii)\ We will only consider the case $t\rightarrow -\infty$. 
The result for $t\rightarrow \infty$ follows by analogous arguments. In the first step, we prove that a SLLN holds for 
$\tilde{S}^{a,d}_t := \int_\R f_{a,d}(t,u)\; d\tilde{S}_u$ with  
\beam\label{Stilde}
\tilde{S}_t := S_t - \E\big(S_t\big)=S_t - \E\big(S_1\big) t\,;
\eeam
i.e., that $\lim_{t\to -\infty} \tilde{S}^{a,d}_t /t = 0$ a.s.. 
The proof is a modification of the proof of \citet[Theorem~3.1]{Fink2011}. 


Without loss of generality we assume that $t<0$. 
By the law of the iterated logarithm (LIL) for L\'evy processes (cf. \citet{Sato1999}, Proposition 48.9) we find a random variable $T$ and a constant $M>0$ such that a.s. for all $s<T$
\begin{equation*}
|\tilde S_s|\leq M(2|s|\log\log |s|)^{\frac{1}{2}}.
\end{equation*} 
We can always make $T$ smaller and so we choose $T<-e$.
For any such path we can assume that $t<T$ and, since \eqref{eq:Riemann}  holds for $\tilde S^{a,d}$, we calculate
\begin{align*}
\tilde S^{a,d}_t 
    & = d \int_{-\infty}^\infty\Big((a+(-s)_+)^{d-1}-(a+(t-s)_+)^{d-1}\Big) \tilde S_s\mathrm{d}s     + d a^{d-1} \int_t^0 \tilde S_s \mathrm{d}s \\
    & = d \int_{-\infty}^t\Big((a+(-s))^{d-1}-(a+(t-s))^{d-1}\Big) \tilde S_s \mathrm{d}s
      + d \int_{t}^0 \Big((a+(-s))^{d-1}- a^{d-1}\Big) \tilde S_s \mathrm{d}s \\
    & \quad + d a^{d-1} \int_t^0 \tilde S_s \mathrm{d}s\,.
\end{align*}
It suffices to show that
    \begin{align}\label{I1}
    \lim_{t\rightarrow-\infty}\frac{1}{|t|}
    \int_{-\infty}^t \Big((a+(-s))^{d-1}-(a+(t-s))^{d-1}\Big) |\tilde S_s| \mathrm{d}s=0\quad a.s.,
    \end{align}
and
    \begin{align}\label{I2}
    \lim_{t\rightarrow-\infty}\frac{1}{|t|}
    \int_{t}^0 (a+(-s))^{d-1} |\tilde S_s| \mathrm{d}s =0\quad a.s.  
    \end{align}

We start with \eqref{I1}. 
Using the LIL we get an upper bound as follows
\begin{align}\label{FLP24c} \nonumber
\lefteqn{\frac{1}{|t|}\int_{-\infty}^t \Big((a+(-s))^{d-1}-(a+(t-s))^{d-1}\Big) |\tilde S_s|\mathrm{d}s }\\ \nonumber
& \leq \frac{M}{|t|}\int_{-\infty}^{-|t|}   \Big((a+(-s))^{d-1}-(a+(t-s))^{d-1}\Big) (2|s|\log\log |s|)^{\frac{1}{2}}\mathrm{d}s\\\nonumber
&= \frac{M}{|t|}\int_{-\infty}^{-|t|-a}   \Big((-v)^{d-1}-(-v-|t|)^{d-1}\Big) (2|v+a|\log\log |v+a|)^{\frac{1}{2}}\mathrm{d}v\\
&= \frac{M|t|}{e|t|}\int_{-\infty}^{-e}\Big((-e^{-1}|t|u -a)^{d-1}-(-|t|-e^{-1}|t|u-a)^{d-1}\Big)
(2(e^{-1}|t||u|)\log\log\big( \frac{|t||u|}{e}\big))^{\frac{1}{2}}\mathrm{d}u,
\end{align}
where we have used in the last line the change of variable $e^{-1}|t|u-a=v$.

Now note that for large $|t|$ and $|u|\geq e$
\begin{align}\label{FLP24b} 
|t||u|\log\log (e^{-1}|t||u|)
 & =|t||u|\log(\log (e^{-1}|t|)+\log|u|)\\ \nonumber
&\leq|t||u|\log\log |t|+|t||u|\log(1+\log|u|)\,.
\end{align}
Combining $(\ref{FLP24b})$ with $|a+b|^{\frac{1}{2}}\leq |a|^{\frac{1}{2}}+|b|^{\frac{1}{2}}$ for $a,b\in\mathbb{R}$ we get an upper bound for $(\ref{FLP24c})$ by
\begin{align}\label{FLP24w}\nonumber
\lefteqn{\frac{M(2e^{-1}|t|\log\log |t|)^{\frac{1}{2}}}{e|t|^{1-d}}
\int_{-\infty}^{-e} \Big((-e^{-1}u-\frac{a}{|t|})^{d-1}-(-1-e^{-1}u-\frac{a}{|t|})^{d-1}\Big) |u|^{\frac{1}{2}}\mathrm{d}u}
\\\nonumber
& \qquad +\frac{M(2e^{-1}|t|)^{\frac{1}{2}}}{e|t|^{1-d}}
\int_{-\infty}^{-e} \Big((-e^{-1}u-\frac{a}{|t|})^{d-1}-(-1-e^{-1}  u-\frac{a}{|t|})^{d-1}\Big)(|u|\log(1+\log |u|))^{\frac{1}{2}}\mathrm{d}u\\\nonumber
&=\frac{M(2e^{-1}\log\log |t|)^{\frac{1}{2}}}{e|t|^{1-(d+\frac{1}{2})}}
\int_{e}^{\infty} \Big((e^{-1}u-\frac{a}{|t|})^{d-1}-(e^{-1}u-1-\frac{a}{|t|})^{d-1} \Big)u^{\frac{1}{2}}\mathrm{d}u
\\
& \qquad +\frac{M(2e^{-1})^{\frac{1}{2}}}{e|t|^{1-(d+\frac{1}{2})}}
\int_{e}^{\infty}\Big((e^{-1}u-\frac{a}{|t|})^{d-1}-(e^{-1}u-1-\frac{a}{|t|})^{d-1}\Big)(u\log(1+\log u))^{\frac{1}{2}}\mathrm{d}u\,.
\end{align}
By a binomial expansion we get 
\[
(e^{-1}u-\frac{a}{|t|}-1)^{d-1}=(e^{-1}u-\frac{a}{|t|})^{d-1}-(d-1)(e^{-1}u-\frac{a}{|t|})^{d-2}+\mathcal{O}(u^{d-3})
\]
and, therefore, (writing $a(u)\sim b(u)$ for $\lim_{u\to\infty} a(u)/b(u) =1$)
\begin{eqnarray*}
\lefteqn{\left[(e^{-1}u-\frac{a}{|t|})^{d-1}-(e^{-1}u-\frac{a}{|t|}a-1)^{d-1}\right](u\log(1+\log |u|))^{\frac{1}{2}}}\nonumber\\
& & \sim (d-1)(e^{-1})^{d-2}u^{d-\frac{3}{2}}(\log\log (u))^{\frac{1}{2}}\,,
\end{eqnarray*}
which ensures the existence of the two integrals in $(\ref{FLP24w})$.
Letting $t\to-\infty$ we obtain \eqref{I1}.\\
 Next we calculate
\[
\frac{1}{|t|}\int_{t}^0 (a+(-s))^{d-1} |\tilde S_s|\mathrm{d}s
     =\frac{1}{|t|}\int_{t}^T (a+(-s))^{d-1} |\tilde S_s|\mathrm{d}s 
  +\frac{1}{|t|}\int_{T}^0 (a+(-s))^{d-1} |\tilde S_s|\mathrm{d}s
\]
The second term tends to zero as $t\to-\infty$, and we consider the first:
\begin{align*}
\frac{1}{|t|}\int_{t}^T (a+(-s))^{d-1} |\tilde S_s|\mathrm{d}s 
 & \leq \frac{M}{|t|}\int_{t}^T (a+(-s))^{d-1} (2|s|\log\log |s|)^{\frac{1}{2}}\mathrm{d}s\\
& \leq \frac{M(2|t|\log\log |t|)^{\frac{1}{2}}}{|t|}\int_{t}^T (a+(-s))^{d-1} \mathrm{d}s\\
&\le   \frac{M(2\log\log |t|)^{\frac{1}{2}}}{|t|^{1/2}}  \int_{t}^T (a-s)^{d-1} \mathrm{d}s\\
&=  \frac{M(2\log\log |t|)^{\frac{1}{2}} (a-t)^d}{d|t|^{\frac{1}{2}}}-\frac{(a-T)^d M(2\log\log |t|)^{\frac{1}{2}}}{d|t|^{\frac{1}{2}}}
\end{align*} 
Letting $t\to-\infty$ we get \eqref{I2}. 
The result follows now from \eqref{Stilde}, which implies that
\[
\int_\R f_{a,d}(t,u) \mathrm{d}S_u = \int_\R f_{a,d}(t,u)d\tilde{S}_u + \E(S_1) \int_\R f_{a,d}(t,u)\mathrm{d}u.
\]
\end{proof}

\begin{proof}[Proof of Proposition \ref{stat_limit}]
For the compact interval $[v,t]$ this is clear. Now consider $v\to-\infty$.  
By the stationarity of the increments of $(S_t^{a,d})_{t\in\R}$ and thus of $(X_t)_{t\in\R}$, we have
    \[
    \int_v^t e^{-(X_t-X_s)}\mathrm{d}s 
        \stackrel{d}{=} \int_v^t e^{-X_{t-s}}\mathrm{d}s 
        = \int_0^{t-v} e^{-X_u}\mathrm{d}u\,,\qquad t>v\,. 
    \]
Hence, we need to show that $\int_0^{t} e^{-X_u}\mathrm{d}u$ converges almost surely to 
a finite random variable as $t\rightarrow\infty$. 
Due to \citet[Theorem 2]{Erickson2005}, this is the case if $X_t\rightarrow \infty$ almost surely as $t\rightarrow\infty$. 
From the definition of $X_t$ we get
 \begin{align*}
     X_t 
        & = \beta_1 t - \alpha_1 S_t^{a,d}
          = \beta_1 t - \alpha_1 \Big(S_t^{a,d} - \E(S_t^{a,d}) +\E(S_t^{a,d}) \Big) \\
        & = t\left(\beta_1 - \alpha_1 \frac{1}{t} \Big(S_t^{a,d} - \E(S_t^{a,d})\Big) - \alpha_1\frac{1}{t} \E(S_1) \int_\R f_{a,d}(t,u)\mathrm{d}u \right)\,.
 \end{align*}
From Proposition~\ref{prop4-modsub} (ii) we know that $\lim_{t\rightarrow\infty}\frac{1}{t} (S_t^{a,d} - \E(S_t^{a,d})) = 0$ a.s.. 
Next we consider the limit $\lim_{t\rightarrow\infty} \frac{1}{t}  \int_\R f_{a,d}(t,u)\mathrm{d}u$. 
To this end we rewrite  the integral
\begin{align*}
 \int_\R f_{a,d}(t,u)\; \mathrm{d}u 
  &=  \int_{-\infty}^0 \left[ (a-u)^d - (a+t-u)^d\right]\; \mathrm{d}u 
    + \int_0^t \left[a^d - (a+t-u)^d \right]\; \mathrm{d}u \\
  &=  \int_{-\infty}^0 t^{d}\left[\left(1-\frac{u-a}{t}\right)^d 
   	- \left(- \frac{u-a}{t}\right)^d\right]\; \mathrm{d}u
   	+ \int_0^t \left[a^d - (a+t-u)^d \right]\; \mathrm{d}u \\
  &=  t^{d+1} \int_{-\infty}^{-a/t} \left[(1-y)^d - (- y)^d\right]\; dy
     + \int_0^t a^{d} -\left[t+a-u\right]^{d}\; \mathrm{d}u \\
  &=:  t^{d+1} c(t) + a^{d}\; t 
    - \frac{1}{d+1} \left(a^{d+1}-(a+t)^{d+1} \right).		
\end{align*}
This implies
    \[
    \lim_{t\rightarrow\infty} \frac{1}{t}  \int_\R f_{a,d}(t,u)\mathrm{d}u
     =  \lim_{t\rightarrow\infty} \left(t^{d} c(t) + a^{d} 
    - \frac{1}{d+1} \left(\frac{a^{d+1}}{t}-\frac{(a+t)^{d+1}}{t} \right)\right)\,,
    \]
which is equal to $a^d$, if $\lim_{t\rightarrow\infty} c(t)$ is finite. But this is
the case, since
    \[
    \lim_{t\rightarrow\infty} c(t)
      = \lim_{t\rightarrow\infty} \int_{-\infty}^{-a/t} \left[(1-y)^d - (- y)^d\right]\; dy
      = \int_{-\infty}^{0} \left[(1-y)^d - (- y)^d\right]\; dy
      = \frac{1}{-1 + d}\,.
     \]
This leads to $\lim_{t\rightarrow \infty}X_t = \infty$, if $\beta_1/\alpha_1 > a^d \E(S_1)$, which proves the result.
\end{proof}

\begin{proof}[Proof of Lemma \ref{stat_sigma_tilde}]
The integral $\int_{-\infty}^t e^{-(X_t-X_s)} \mathrm{d}s$ is well defined due to Proposition \ref{stat_limit}. Hence, we can use the Cram\'er-Wold device and calculate for $a_1,\ldots,a_m\in\R$ and $m\in\N$,
\beao
\frac1{\alpha_0\beta_1}\sum_{i=1}^m a_i \widetilde\sigma^2_{t_i+h}
 &=& \sum_{i=1}^m a_i  \int_{-\infty}^{t_i+h}  \exp\Big\{-\beta_1(t_i+h-s)+\alpha_1( S_{t_i+h}^{a,d}-S_s^{a,d}) \Big\} \mathrm{d}s\\
  &=& \sum_{i=1}^m a_i \int_{-\infty}^{t_i}  \exp\Big\{-\beta_1(t_i-v)+\alpha_1( S_{t_i+h}^{a,d}-S_{v+h}^{a,d}) \Big\} \mathrm{d}v\\
  &\eqd& \sum_{i=1}^m a_i \int_{-\infty}^{t_i}  \exp\Big\{-\beta_1(t_i-v)+\alpha_1( S_{t_i}^{a,d}-S_{v}^{a,d}) \Big\} \mathrm{d}v\\
  &=& \frac1{\alpha_0\beta_1} \sum_{i=1}^m a_i \widetilde\sigma^2_{t_i},
\eeao
by the stationary increments of $(S^{a,d}_t)_{t\in\R}$.
\end{proof}

\subsection*{Proofs for Section \ref{s4}}

\begin{proof}[Proof of Proposition \ref{prop1-modsub}]
(i), (ii) and (iii) are obvious; (iv) is based on the fact that for all $t\in\R$, by a Taylor expansion,
\beao
\lim_{s\to -\infty} \frac{|f_{a,d}(t,s)|}{|s|^{d-1}} 
= \lim_{s\to -\infty} \frac{|(a-s)^d - (t+a-s)^d|}{|s|^{d-1}}  
= \lim_{s\to -\infty} \frac{|dt (a-s)^{d-1}|}{|s|^{d-1}}  = |d t|. 
\eeao
(v) From (iii) follows that it suffices to check for which $\delta>0$ the integral 
$\int_{-\infty}^N |s|^{\delta(d-1)} \mathrm{d}s<\infty$ for some $N<0$. This is exactly the case for $\delta>1/|d-1|$.
\end{proof}

\begin{proof}[Proof of Proposition \ref{prop2-modsub}]

To show that \eqref{frac_sub_modmvn} exists as an $L^2(\Omega,P)$ limit of approximating step functions, we apply Theorem 3.3 in \citet{Rajput1989}. It follows that \eqref{frac_sub_modmvn} exists for all $t\in\R$ in the $L^2$-sense if
\begin{align*}
	\int_{\R} & 
	    \left[ f_{a,d}(t,s) \gamma_S 
	            + \int_{\R} f_{a,d}(t,s) x 
								\left[\mathds{1}_{\{|f_{a,d}(t,s)x|\le 1\}}-\mathds{1}_{\{|x|\le 1\}}\right] \;
								\nu_S(dx) 
				+ \int_\R (f_{a,d}(t,s)x)^2\;\nu_S(dx)\right]\;ds 
		< \infty
\end{align*}
Since $\gamma_S = \E(S_1) - \int_{|x|>1}x\;\nu_S(dx)$, we find that
\begin{align*}
	\int_{\R} & \left[ f_{a,d}(t,s)\;  \E(S_1) 
			    + \int_{\R} f_{a,d}(t,s) x \left[\mathds{1}_{\{|f_{a,d}(t,s)x|\le 1\}}-1\right] \;\nu_S(dx) 
			    + \int_\R (f_{a,d}(t,s)x)^2\;\nu_S(dx)\right]\;ds\\
    		& \leq \E(S_1) \int_\R f_{a,d}(t,s)\;ds 
    		    +\int_\R\int_{\R} f_{a,d}(t,s)x \mathds{1}_{\{|f_{a,d}(t,s)x|> 1\}}\;\nu_S(dx)\; ds
				+ \int_\R (f_{a,d}(t,s)x)^2\;\nu_S(dx)\;ds\\	
			&\le \E(S_1)\; \|f_{a,d}(t,\cdot)\|_{L^1(\R)} + 2\int_\R\int_{\R}(f_{a,d}(t,s)x)^2 \;\nu_S(dx)\; ds	\\
			&= \E(S_1)\; \|f_{a,d}(t,\cdot)\|_{L^1(\R)} 
			    + 2 \|f_{a,d}(t,\cdot)\|^2_{L^2(\R)} \int_{\R} x^2\;\nu_S(dx)
			 < \infty\,,
\end{align*}
where the last expression is finite due to Proposition \ref{prop1-modsub} (v) and $\E(S_1^2)<\infty$.
\end{proof}

\begin{proof}[Proof of Lemma \ref{lem:fquad}] By substituting $y:={(u-a)}/{t}$ we obtain
  \begin{align*}
  \int_\R f^2_{a,d}(t,u)\; \mathrm{d}u 
  &= \int_{-\infty}^0 \left[ (a-u)^d - (a+t-u)^d\right]^2\; \mathrm{d}u 
    + \int_0^t \left[a^d - (a+t-u)^d \right]^2\; \mathrm{d}u \\
  &= \int_{-\infty}^0 t^{2d}\left[\left(1-\frac{u-a}{t}\right)^d 
   	- \left(- \frac{u-a}{t}\right)^d\right]^2\; \mathrm{d}u
   	+ \int_0^t \left[a^d - (a+t-u)^d \right]^2\; \mathrm{d}u \\
  &= t^{2d+1} \int_{-\infty}^{-a/t} \left[(1-y)^d - (- y)^d\right]^2\; dy
     + \int_0^t a^{2d} - 2a^{d}\left[t+a-u\right]^d +\left[t+a-u\right]^{2d}\; \mathrm{d}u \\
  &= t^{2d+1} c(t) + a^{2d}\; t +\frac{2a^d}{d+1} \left[t+a-u\right]^{d+1}\Big|_0^t
    - \frac{1}{2d+1} \left[t+a-u\right]^{2d+1}\Big|_0^t.							  
    \end{align*}
Further note that for the normalization constant $1/\Gamma(d+1)$ in \eqref{mvnkernel} we obtain
  \begin{align*}
	\int_{-\infty}^1 \left[(1-y)^d_+ - (-y)^d_+\right]^2\; dy 
	&= \Gamma(d+1) \int_{-\infty}^1 \left(\mvn(1,y)\right)^2\; dy 
	= \frac{\Gamma(d+1)}{\Gamma(2d+2)\sin(\pi(d+0.5))}.
  \end{align*}
Consequently,
  \begin{align*}
	c(t) 
	&= \frac{\Gamma(d+1)}{\Gamma(2d+2)\sin(\pi(d+0.5))}
	  -\int_{-a/t}^0\left[(1-y)^d-(-y)^d\right]^2\; dy
	  - \int_0^1 (1-y)^{2d}\; dy \\
	&= \frac{\Gamma(d+1)}{\Gamma(2d+2)\sin(\pi(d+0.5))}  
	-\int_{-a/t}^0\left[(1-y)^d-(-y)^d\right]^2\; dy
	+ \frac{1}{2d+1}.			 
\end{align*}
Since $\int_{-a/t}^0\left[(1-y)^d-(-y)^d\right]^2\; dy \to 0$ as $t\to\infty$, the assertion holds.
\end{proof}

\begin{proof}[Proof of Proposition \ref{prop5-modsub}]
We use the notation as in \eqref{Stilde},
$$\tilde{S}_t := S_t - \E\big(S_t\big)\quad\mbox{and}\quad \tilde{S}^{a,d}_t := \int_\R f_{a,d}(t,u)\; d\tilde{S}_u.$$
For $t,s\ge 0$ we calculate
\begin{align*}
	\cov(S^{a,d}_t, S^{a,d}_s) 
	&= \cov(\tilde{S}^{a,d}_t, \tilde{S}^{a,d}_s) 
	= \E\big(\tilde{S}^{a,d}_t \tilde{S}^{a,d}_s\big) 
	= \frac{1}{2}\left( \E\big((\tilde{S}^{a,d}_t)^2\big)  
	 + \E\big((\tilde{S}^{a,d}_s)^2\big)
	 - \E\big((\tilde{S}^{a,d}_t - \tilde{S}^{a,d}_s)^2\big)\right) \\
	 &= \frac{1}{2}\left( \E\big((\tilde{S}^{a,d}_t)^2\big)  
	+ \E\big((\tilde{S}^{a,d}_s)^2\big)																		- \E\big((\tilde{S}^{a,d}_{t-s})^2\big)\right).
\end{align*}
In the last step we have used that the increments are stationary. 
Furthermore, 
\begin{align*}
	\E\big((\tilde{S}^{a,d}_t)^2\big)
	   = \var\big(\tilde{S}^{a,d}_t \big)  
	   = \var\left(S_1\right) \int_\R f^2_{a,d}(t,u)\; \mathrm{d}u,
\end{align*}
such that by the linearity of the covariance operator,
\begin{align*}
\gamma_\Delta(h) &=  \cov(S^{a,d}_{s+(h+1)\Delta}, S^{a,d}_{s+\Delta})
										- \cov(S^{a,d}_{s+(h+1)\Delta}, S^{a,d}_{s})
										- \cov(S^{a,d}_{s+h\Delta}, S^{a,d}_{s+\Delta}) 
									 + \cov(S^{a,d}_{s+h\Delta}, S^{a,d}_{s}) \\					
&= \frac{1}{2} \left( \E\big((\tilde{S}^{a,d}_{(h+1)\Delta})^2\big)  + \E\big((\tilde{S}^{a,d}_{(h-1)\Delta})^2\big)- 2 \E\big((\tilde{S}^{a,d}_{h\Delta})^2\big)\right) \\
&= \frac{1}{2} \var\left(S_1\right)  
\left[ \int_\R f^2_{a,d}\big((h+1)\Delta,u\big)\; \mathrm{d}u + \int_\R f^2_{a,d}\big((h-1)\Delta,u\big)\; \mathrm{d}u							- 2 \int_\R f^2_{a,d}(h\Delta,u)\; \mathrm{d}u\right].
\end{align*}
Now, using Lemma \ref{lem:fquad} we obtain
\begin{align*}
	\gamma_\Delta(h) &= \frac{1}{2} \var\left(S_1\right)
												\Big[ -\frac{2a^d}{d+1}\left(((h\Delta+a)+\Delta)^{d+1} +  ((h\Delta+a)-\Delta)^{d+1} - 2 (h\Delta+a)^{d+1} \right) \\
&\qquad\qquad\qquad+\frac{1}{2d+1} \left( ((h\Delta+a)+\Delta)^{2d+1} + ((h\Delta+a)-\Delta)^{2d+1} - 2 (h\Delta+a)^{2d+1} \right) \\
&\qquad\qquad\qquad+ c(h\Delta+\Delta) (h\Delta+\Delta)^{2d+1} + c(h\Delta-\Delta) (h\Delta-\Delta)^{2d+1} - 2c(h\Delta) (h\Delta)^{2d+1}\Big],
\end{align*}
where $c(t)$ is defined as in \eqref{eq:fc} and according to \eqref{eq:asymptCt} converges for $t\to\infty$ to a positive constant, which we denote by $c$. 
Consequently, a Taylor expansion gives for $h\to\infty$,
\begin{align*}
	\gamma_\Delta(h) &= \frac{1}{2} \var\left(S_1\right)
	               \Big[ -\frac{2a^d}{d+1}(h\Delta+a)^{d+1} \left( (d+1)d\frac{\Delta^2}{(h\Delta+a)^2} + 
														\mathcal O\left(\frac{1}{(h\Delta+a)^4}\right)\right) \\
	&\qquad\qquad\qquad\qquad\qquad+ \frac{(h\Delta+a)^{2d+1}}{2d+1} \left((2d+1)2d\frac{\Delta^2}{(h\Delta+a)^2} +
														\mathcal O\left(\frac{1}{(h\Delta+a)^4}\right)\right) \\
	&\qquad\qquad\qquad\qquad\qquad+ c \left( (2d+1)2d\frac{\Delta^2}{(h\Delta)^2} +
														\mathcal O\left(\frac{1}{(h\Delta)^4}\right)\right)
									\Big] \\
&\sim \var\left(S_1\right) (-d) a^d  \Delta^2 \; (h\Delta+a)^{d-1}.
\end{align*} 
\end{proof}

\section*{Acknowledgements}
We thank Aleksey Min from the Chair of Mathematical Finance at the Technical University of Munich for access to the Chair's Thomas Reuters database. Further we would like to thank Thiago do R\^{e}go Sousa for interesting discussions and useful comments on the simulation-based version of the generalised method of moments.

\medskip

\bibliographystyle{chicago}

\bibliography{fracCOGARCH_bib}

\vskip 1cm
\noindent
TECHNICAL UNIVERSITY OF MUNICH\\
DEPARTMENT OF MATHEMATICS\\
85748 GARCHING, GERMANY.
\vskip 2pt
\noindent
E-MAIL: \{\;haug\,,\,cklu\;\}@tum.de\\
URL: \url{http://www.statistics.ma.tum.de}

\end{document}